\documentclass{amsart}

\usepackage{amsmath,amsthm,amstext,amssymb,bbm}
\usepackage{hyperref}
\usepackage{accents}
\usepackage[utf8]{inputenc}
\usepackage{enumerate,enumitem}
\usepackage{amscd,  amsfonts,  amsbsy,  mathrsfs, upgreek, mathtools, stmaryrd, bbm}

\theoremstyle{plain}
\newtheorem{theorem}{Theorem}[section]
\newtheorem{lemma}[theorem]{Lemma}

\newtheorem*{observation*}{Observation}

\newtheorem{proposition}[theorem]{Proposition}

\newtheorem*{claim*}{Claim}
\newtheorem*{subclaim*}{Subclaim}
\theoremstyle{definition}
\newtheorem{definition}[theorem]{Definition}

\newtheorem{question}[theorem]{Question}

\newcommand{\betrag}[1]{\vert{#1}\vert}

\newcommand{\lub}{{\rm{lub}}}

\newcommand{\cof}[1]{{{\rm{cof}}(#1)}}
\newcommand{\otp}[1]{{{\rm{otp}}\left(#1\right)}}

\newcommand{\POT}[1]{{\mathcal{P}}({#1})}

\newcommand{\map}[3]{{#1}:{#2}\longrightarrow{#3}}

\newcommand{\Set}[2]{\{{#1}~\vert~{#2}\}}
\newcommand{\seq}[2]{\langle{#1}~\vert~{#2}\rangle}

\newcommand{\anf}[1]{{\text{``}\hspace{0.3ex}{#1}\hspace{0.3ex}\text{''}}}

\newcommand{\HH}[1]{{\rm{H}}(#1)}
\newcommand{\Ult}[2]{{\mathrm{Ult}}({#1},{#2})}

\newcommand{\Col}[2]{{\rm{Col}}({#1},{#2})}

\newcommand{\Lim}{{\rm{Lim}}}

\newcommand{\LL}{{\rm{L}}}

\newcommand{\ZFC}{{\rm{ZFC}}}

\newcommand{\PPP}{{\mathbb{P}}}

\newcommand{\RRR}{{\mathbb{R}}}

\newcommand{\VV}{{\rm{V}}}
\newcommand{\WW}{{\rm{W}}}

 \newcommand{\GCH}{{\rm{GCH}}}
  \newcommand{\SCH}{{\rm{SCH}}}
\makeatletter
\@namedef{subjclassname@2020}{%
  \textup{2020} Mathematics Subject Classification}
\makeatother

\theoremstyle{remark}
\newtheorem{case}{Case}

\title{Closure properties of measurable ultrapowers}

\author{Philipp L\"ucke}
\address{Institut de Matem\`{a}tica, Universitat de Barcelona. 
Gran via de les Corts Catalanes 585,
08007 Barcelona, Spain.}
\author{Sandra M\"uller}
\address{Institut f\"ur Mathematik, Universit\"at Wien. Kolingasse 14-16, 1090 Wien, Austria.}


\thanks{The authors would like to thank Peter Koepke for a discussion that motivated the work presented in this paper. 
   This project has received funding from the European Union’s Horizon 2020 research and innovation programme under the Marie Sk{\l}odowska-Curie grant agreement No 842082 of the first author (Project \emph{SAIFIA: Strong Axioms of Infinity -- Frameworks, Interactions and Applications}). 
   During the preparation of this paper, the first author was partially supported by the Deutsche Forschungsgemeinschaft (DFG, German Research Foundation) under Germany's Excellence Strategy -- EXC-2047/1 -- 390685813. 
   The second author gratefully acknowledges funding from L'OR\'{E}AL Austria, in collaboration with the Austrian UNESCO Commission and in cooperation with the Austrian Academy of Sciences - Fellowship \emph{Determinacy and Large Cardinals}. 
 }

\begin{document}

\begin{abstract}
 We study closure properties of measurable ultrapowers with respect to Hamkin's notion of \emph{freshness} and show that the extent of these properties highly depends on the combinatorial properties of the underlying model of set theory. 
 In one direction, a result of Sakai shows that, by collapsing a strongly compact cardinal to become the double successor of a measurable cardinal, it is possible to obtain a model of set theory in which such ultrapowers possess the strongest possible closure properties. 
 In the other direction, we use various square principles to show that measurable ultrapowers of canonical inner models only possess the minimal amount of closure properties. 
  In addition, the techniques developed in the proofs of these results also allow us to derive statements about the consistency strength of the existence of measurable ultrapowers  with non-minimal closure properties.  
\end{abstract}

\maketitle

%\setcounter{tocdepth}{1}
%\tableofcontents

%%%%%%%%%%%%%%%%%%%%%%%%%%%%%%%%%%%%%%%%
%%%%%%%%%%%%%%%%%%%%%%%%%%%%%%%%%%%%%%%%
%%%%%%%%%%%%%%%%%%%%%%%%%%%%%%%%%%%%%%%%

\section{Introduction}

 The present paper studies the structural properties of ultrapowers of models of set theory constructed with the help of normal ultrafilters on measurable cardinals. 
 Two of the most fundamental properties of these ultrapowers are that these models do not contain the ultrafilter utilized in their construction and that they are closed under sequences of length equal to the relevant measurable cardinal. 
  In the following, we want to further analyze the closure and non-closure properties of measurable  ultrapowers through the following notion introduced by Hamkins in \cite{zbMATH01158349}.

\begin{definition}[Hamkins]
 Given a class $M$, a set $A$ of ordinals is \emph{fresh over $M$} if $A\notin M$ and $A\cap\alpha\in M$ for all $\alpha<\lub(A)$.\footnote{Here $\lub(A)$ denotes the \emph{least upper bound} of $A$.}
\end{definition}

Given a normal ultrafilter $U$ on a measurable cardinal, we let $\Ult{\VV}{U}$ denote the (transitive collapse of the) induced ultrapower and we let $\map{j_U}{\VV}{\Ult{\VV}{U}}$ denote the corresponding elementary embedding. For notational simplicity, we confuse $\Ult{\VV}{U}$ and its elements with their transitive collapses.
 In this paper, for a given normal ultrafilter $U$, we aim to determine the class of limit ordinals containing an unbounded subset that is fresh over the ultrapower $\Ult{\VV}{U}$. 
  For the images of regular cardinals under the embedding $j_U$, this question was already studied by Shani in \cite{MR3523658}. Moreover, Sakai investigated closure properties of measurable ultrapowers that imply the non-existence of unbounded fresh subsets at many ordinals in \cite{sakainote}. 

%  \todo{Short summary?}
 
 The following proposition lists the obvious closure properties of measurable ultrapowers with respect to the non-existence of fresh subsets. 
  Note that the second part of the second statement also follows directly from {\cite[Corollary 3.3]{sakainote}}. 
   The proof of this proposition and the next one will be given in Section \ref{section:2}.

\begin{proposition}\label{proposition:MinNonFresh}
 Let $U$ be a normal ultrafilter on a measurable cardinal $\delta$ and let $\lambda$ be a limit ordinal. 
 \begin{enumerate}
  \item If the cardinal $\cof{\lambda}$ is either smaller than $\delta^+$ or weakly compact, then no unbounded subset of $\lambda$ is fresh over $\Ult{\VV}{U}$. 
  
  \item If $\cof{\lambda}>2^\delta$ and there exists a ${<}(2^\delta)^+$-closed ultrafilter on $\cof{\lambda}$ that contains all cobounded subsets of $\cof{\lambda}$, then no unbounded subset of $\lambda$ is fresh over $\Ult{\VV}{U}$. In particular, if there exists a strongly compact cardinal $\kappa$ with $\delta<\kappa\leq\cof{\lambda}$, then no unbounded subset of $\lambda$ is fresh over $\Ult{\VV}{U}$. 
 \end{enumerate}
\end{proposition}

In the other direction, the fact that normal ultrafilters are not contained in the corresponding ultrapowers directly yields the following non-closure properties of these ultrapowers.

\begin{proposition}\label{proposition:MinFresh}
 Let $U$ be a normal ultrafilter on a measurable cardinal $\delta$. 
 \begin{enumerate}
  \item If $\kappa>\delta$ is the minimal cardinal with $\POT{\kappa}\nsubseteq\Ult{\VV}{U}$, then there is an unbounded subset of $\kappa$ that is fresh over $\Ult{\VV}{U}$. 
  
    \item If $\lambda$ is a limit ordinal with $\cof{\lambda}^{\Ult{\VV}{U}}=j_U(\delta^+)$, then there is an unbounded subset of $\lambda$ that is fresh over $\Ult{\VV}{U}$. 
    
  \item  If $2^\delta=\delta^+$ holds and $\lambda$ is a limit ordinal with $\cof{\lambda}=\delta^+$, then there is an unbounded subset of $\lambda$ that is fresh over $\Ult{\VV}{U}$. 
 \end{enumerate}
\end{proposition}

In the following, we will present results that show that the above propositions already cover all provable closure and non-closure properties of measurable ultrapowers, in the sense that there are models of set theory in which fresh subsets exist at all limit ordinals that are not ruled out by Proposition \ref{proposition:MinNonFresh} and models in which fresh subsets only exist at limit ordinals, where their existence is guaranteed by Proposition \ref{proposition:MinFresh}. 
 
 Results of Sakai on the approximation properties of measurable ultrapowers in \cite{sakainote} can directly be used to prove the following result that shows that models with minimal non-closure properties can be constructed by collapsing a strongly compact cardinal to become the double successor of a measurable cardinal. 
 Note that we phrase the following result in a non-standard way to clearly distinguish between the ground model of the  forcing extension and the model used in the corresponding ultrapower construction.

\begin{theorem}\label{theorem:sakai}
 Let $\delta$ be a measurable cardinal and let $\WW$ be an inner model such that the $\GCH$ holds in $\WW$, $\VV$ is a $\Col{(\delta^+)^\VV}{{<}(\delta^{++})^\VV}^\WW$-generic extension of $\WW$ and $(\delta^{++})^\VV$ is strongly compact in $\WW$.  
  Given a normal ultrafilter $U$ on $\delta$, the following statements are equivalent for every  limit ordinal $\lambda$:  
 \begin{enumerate}
   \item There is an unbounded subset of $\lambda$ that is fresh over $\Ult{\VV}{U}$. 
   
   \item $\cof{\lambda}=\delta^+$. 
  \end{enumerate}
\end{theorem}

\begin{proof}
 For one direction, our assumptions directly imply that the $\GCH$ holds and therefore Proposition \ref{proposition:MinFresh} shows that for every limit ordinal $\lambda$ with $\cof{\lambda}=\delta^+$, there exists an unbounded subset of $\lambda$ that is fresh over $\Ult{\VV}{U}$. 
  For the other direction, assume, towards a contradiction, that $\lambda$ is a limit ordinal with $\cof{\lambda}\neq\delta^+$ and the property that some unbounded subset $A$ of $\lambda$ is fresh over $\Ult{\VV}{U}$. Then Proposition \ref{proposition:MinNonFresh} implies that $\cof{\lambda}>\delta^+$. 
  By our assumption, {\cite[Corollary 3.3]{sakainote}} directly implies that $\Ult{\VV}{U}$ has the \emph{$\delta^{++}$-approximation property}, i.e., we have $X\in\Ult{\VV}{U}$ whenever a set $X$ of ordinals has the property that $X\cap a\in\Ult{\VV}{U}$ holds for all $a\in\Ult{\VV}{U}$ with $\betrag{a}^{\Ult{\VV}{U}}<\delta^{++}$. 
  In particular, there exists some $a\in\Ult{\VV}{U}$ with $\betrag{a}^{\Ult{\VV}{U}}<\delta^{++}$ and $A\cap a\notin\Ult{\VV}{U}$. 
  But then the fact that $\cof{\lambda}\geq\delta^{++}>\betrag{a}^{\Ult{\VV}{U}}\geq\betrag{A\cap a}$ implies that $A\cap a$ is bounded in $\lambda$ and hence the assumption that $A$ is fresh over $\Ult{\VV}{U}$ allows us to conclude that $A\cap a$ is an element of $\Ult{\VV}{U}$, a contradiction.  
\end{proof}

For the other direction, we will prove results that show that canonical inner models provide measurable ultrapowers with closure properties that are minimal in the above sense. These arguments make use of the validity of various combinatorial principles in these models. In particular, they heavily rely on the existence of suitable \emph{square sequences}. 
 For specific types of cardinals, similar constructions have already been done in {\cite[Section 3.3]{sakainote}} and {\cite[Section 3]{MR3523658}}.

\begin{definition}
 \begin{enumerate}
  \item Given an uncountable regular cardinal $\kappa$, a sequence $\seq{C_\gamma}{\gamma\in\Lim\cap\kappa}$ is a \emph{$\square(\kappa)$-sequence} if the following statements hold: 
   \begin{enumerate}
    \item $C_\gamma$ is a closed unbounded subset of $\gamma$ for all $\gamma\in\Lim\cap\kappa$. 
    
    \item If $\gamma\in\Lim\cap\kappa$ and $\beta\in\Lim(C_\gamma)$, then $C_\beta=C_\gamma\cap\beta$. 
    
    \item There is no closed unbounded subset $C$ of $\kappa$ with $C\cap\gamma=C_\gamma$ for all $\gamma\in\Lim(C)$. 
   \end{enumerate}
   
  \item Given an infinite cardinal $\kappa$, a $\square(\kappa^+)$-sequence $\seq{C_\gamma}{\gamma\in\Lim\cap\kappa^+}$ is a \emph{$\square_\kappa$-sequence} if $\otp{C_\gamma}\leq\kappa$ holds for all $\gamma\in\Lim\cap\kappa^+$. 
 \end{enumerate}
\end{definition}

 The next result shows that, in certain models of set theory, fresh subsets for measurable ultrapowers exist at all limit ordinals that are not ruled out by the conclusions of Proposition \ref{proposition:MinNonFresh}.

\begin{theorem}\label{theorem:CanonicalModelsCharFresh}
  Let $U$ be a normal ultrafilter on a measurable cardinal $\delta$.
  % satisfying \todo[fancyline]{Weg!} $2^\delta=\delta^+$. 
  Assume that the following statements hold: 
 \begin{enumerate}
  \item[(a)]  The $\GCH$ holds at all cardinals greater than or equal to $\delta$. %\emph{Singular Cardinal Hypothesis} \todo[fancyline]{$\GCH$} $\SCH$.\footnote{Remember that the $\SCH$ states that $\kappa^{\cof{\kappa}}=\kappa^+$ holds for every singular cardinal $\kappa$ with $2^{\cof{\kappa}}<\kappa$.} 
  
  \item[(b)] If $\kappa>\delta^+$ is a regular cardinal that is not weakly compact, then there exists a $\square(\kappa)$-sequence. 
  
  \item[(c)] If $\kappa>\delta$ is a singular cardinal, then there exists a $\square_\kappa$-sequence. 
  
%  \item[(d)] If $\kappa>\delta$ is a singular cardinal, then $\Diamond(E)$ holds for every stationary subset $E$ of $\kappa^+$. 
 \end{enumerate}
 Then the following statements are equivalent for every limit ordinal $\lambda$: 
  \begin{enumerate}
   \item There is an unbounded subset of $\lambda$ that is fresh over $\Ult{\VV}{U}$. 
   
   \item The cardinal $\cof{\lambda}$ is greater than $\delta$ and not weakly compact. 
  \end{enumerate}
\end{theorem}

With the help of results of Schimmerling and Zeman in \cite{MR2081183} and \cite{MR2563821}, the above result will allow us to show that measurable ultrapowers of a large class of canonical inner models, so called \emph{Jensen-style extender models}, possess the minimal amount of closure properties with respect to freshness. These inner models go back to Jensen in \cite{Je97}, following a suggestion of S. Friedman, and can have various large cardinals below supercompact cardinals. As for example in \cite{MR2563821}, we demand that they satisfy classical consequences of iterability such as solidity and condensation. The theorem below holds for Mitchell-Steel extender models with the same properties constructed as in \cite{MS94} as well. But it turned out that Jensen-style constructions are more natural in the proof of $\square$-principles in canonical inner models, so this is what Schimmerling and Zeman use in \cite{MR2081183} and \cite{MR2563821}, and we decided to follow their notation.
 %
%Note that results of  show that, in a large class of canonical inner models (i.e. fine structural extender model with $\lambda$-indexing of extenders that do not contain successor cardinals $\kappa^+$ with the property that every pair of stationary subsets of $S^{\kappa^+}_{{<}\kappa}$ reflects at a common point of cofinality less than $\kappa$), the principles $\square(\kappa)$ holds for some uncountable regular cardinal $\kappa$ if and only if $\kappa$ is not weakly compact. 

\begin{theorem}\label{theorem:InnerModels}
 Assume that $\VV$ is a \emph{Jensen-style extender model} that does not have a subcompact cardinal. 
 Then the statements (i) and (ii) listed in Theorem \ref{theorem:CanonicalModelsCharFresh} are equivalent for every normal ultrafilter $U$ on a measurable cardinal $\delta$ and every limit ordinal $\lambda$.  
\end{theorem}

We restrict ourselves to inner models without subcompact cardinals in the statement of Theorem \ref{theorem:InnerModels}, as the non-existence of $\square_\kappa$-sequences in Jensen-style extender models is equivalent to $\kappa$ being  subcompact (see \cite{MR2081183}). Results of Kypriotakis and Zeman in \cite{zbMATH06136601} show that $\square(\kappa^+)$-sequences can exists even if $\kappa$ is subcompact, but we decided to not discuss this further here.

The techniques developed in the proof of Theorem \ref{theorem:CanonicalModelsCharFresh} also allows us to derive large lower bounds for the consistency strength of the conclusion of Theorem \ref{theorem:sakai}.

\begin{theorem}\label{theorem:StrengthSakai}
 Let $U$ be a normal ultrafilter on a measurable cardinal $\delta$. 
 \begin{enumerate}
  \item If there exists a regular cardinal $\kappa>\delta^+$ and a cardinal $\theta\in\{\kappa,\kappa^+\}$ with the property that  there  exists a $\square(\theta)$-sequence, then there exists a limit ordinal $\lambda$ of cofinality $\theta$ and an unbounded subset of $\lambda$ that is fresh over $\Ult{\VV}{U}$. 
  
  \item If there exists a singular strong limit cardinal $\kappa$ with the property that $\cof{\kappa}=\delta$, the $\SCH$\footnote{Remember that the $\SCH$ states that $\kappa^{\cof{\kappa}}=\kappa^+$ holds for every singular cardinal $\kappa$ with $2^{\cof{\kappa}}<\kappa$.} holds at $\kappa$ and there exists a $\square_\kappa$-sequence, then there exists a limit ordinal $\lambda$ with $\cof{\lambda}=\kappa^+$ and an unbounded subset of $\lambda$ that is fresh over $\Ult{\VV}{U}$. 
%    $\kappa>\delta^+$ is  with the property that all unbounded subsets of $j_U(\kappa)$ and $j_U(\kappa^+)$ are not fresh over $\Ult{\VV}{U}$, then there are no $\square(\kappa)$-sequences and no $\square(\kappa^+)$-sequences. 
 \end{enumerate}
  % with the property that the statements (i) and (ii) listed in \ref{theorem:sakai} are equivalent for every limit ordinal $\lambda$. 
  %
% If $\kappa$ is a singular strong limit cardinal of cofinality $\delta$, then there exists  no $\square_\kappa$-sequence.  
\end{theorem}

If $U$ is a normal ultrafilter on a measurable cardinal $\delta$ with the property that the statements (i) and (ii) listed in \ref{theorem:sakai} are equivalent for every limit ordinal $\lambda$, then the first part of the above theorem shows that $\kappa=\delta^{++}$ is a countably closed\footnote{Remember that a cardinal $\kappa$ is \emph{countably closed} if $\mu^\omega<\kappa$ holds for all cardinals $\mu<\kappa$.} regular cardinal that is greater than $\max\{2^{\aleph_0},\aleph_3\}$ and has the property that there are no $\square(\kappa)$- and no $\square_\kappa$-sequences. 
 By {\cite[Theorem 5.6]{SchimmerlingCoherentSeq}}, the existence of such a cardinal implies \emph{Projective Determinacy}. 
 In addition, {\cite[Theorem 0.1]{zbMATH05532622}} derives the existence of a sharp for a proper class model with a proper class of strong cardinals and a proper class of Woodin cardinals from the existence of such a cardinal. 
 Moreover, note that the results of \cite{zbMATH05051863} show that the existence of a singular strong limit cardinal $\kappa$ with the property that there are no $\square_\kappa$-sequences implies that $\mathrm{AD}$ holds in $\LL(\RRR)$, and even stronger consequences of this assumptions are given by the results of \cite{doi:10.1142/S0219061314500032}. 
  Finally, note that work of Gitik and Mitchell (see \cite{zbMATH00429348} and \cite{zbMATH00034081}) shows that a failure of $\SCH$ at a singular strong limit cardinal $\kappa$ of uncountable cofinality implies that $\kappa$ is a measurable cardinal of high Mitchell order in a canonical inner model.

 Finally, our techniques also allow us to determine the exact consistency strength of the existence of a measurable ultrapower that has the property that no unbounded subsets of the double successor of the corresponding  measurable cardinal are fresh over it. 
 This result is motivated by results of Cummings in \cite{MR1217188} that determine the exact consistency strength of the existence of a measurable ultrapower that contains the power set of the successor of the corresponding  measurable cardinal. 
 In our setting, Cummings' results can be rephrased in the following way:

 \begin{theorem}
  The following statements are equiconsistent over the theory $\ZFC$: 
  \begin{enumerate}
   \item There exists a $(\delta+2)$-strong cardinal $\delta$. 
  
   \item There exists a normal ultrafilter $U$ on a measurable cardinal $\delta$ with the property that no unbounded subset of $\delta^+$ is fresh over $\Ult{\VV}{U}$. 
  \end{enumerate}
 \end{theorem}
 
 \begin{proof}
   In one direction, {\cite[Theorem 1]{MR1217188}} shows that, starting with a model of $\ZFC$ containing a $(\delta+2)$-strong cardinal $\delta$, it is possible to construct  a model in which there exists a normal ultrafilter $U$ on $\delta$ satisfying $\POT{\delta^+}\subseteq\Ult{\VV}{U}$. In particular, no subset of $\delta^+$ is fresh over $\Ult{\VV}{U}$ in this model. 
 In the other direction, if $U$ is a normal ultrafilter on a measurable cardinal $\delta$ with the property that no unbounded subset of $\delta^+$ is fresh over $\Ult{\VV}{U}$, then the closure of  $\Ult{\VV}{U}$ under $\kappa$-sequences implies that  $\POT{\delta^+}\subseteq\Ult{\VV}{U}$ holds and hence {\cite[Theorem 2]{MR1217188}} yields an inner model with a $(\delta+2)$-strong cardinal $\delta$.  
 \end{proof}

The following theorem determines the exact consistency of the corresponding statement for double successors of measurable cardinals.

\begin{theorem}\label{theorem:ConsStrength}
 The following statements are equiconsistent over the theory $\ZFC$: 
 \begin{enumerate}
  \item There exists a weakly compact cardinal above a measurable cardinal. 
  
  \item There exists a normal ultrafilter $U$ on a measurable cardinal $\delta$ with the property that no unbounded subset of $\delta^{++}$ is fresh over $\Ult{\VV}{U}$. 
  %
%  \item There exists a normal ultrafilter $U$ on a measurable cardinal $\delta$ and a limit ordinal $\lambda$ with $\cof{\lambda}>\delta$ and the property that no unbounded subset of $\lambda$ is fresh over $\Ult{\VV}{U}$. \todo[fancyline]{Weg!}
 \end{enumerate}
\end{theorem}

%%%%%%%%%%%%%%%%%%%%%%%%%%%%%%

\section{Simple closure and non-closure properties}\label{section:2}

 In this section, we prove the two propositions stated in the introduction.

\begin{proof}[Proof of Proposition \ref{proposition:MinNonFresh}]
 (i) If $\cof{\lambda}\leq\delta$, then the desired statement follows directly from the closure of $\Ult{\VV}{U}$ under $\delta$-sequences. Hence, we may assume that $\kappa=\cof{\lambda}$ is a weakly compact cardinal greater than $\delta$. Pick a cofinal sequence $\seq{\gamma_\alpha}{\alpha<\kappa}$ in $\lambda$ and fix an unbounded subset $A$ of $\lambda$ such that $A\cap\gamma\in\Ult{\VV}{U}$ for all $\gamma<\lambda$. Given $\alpha<\kappa$, fix functions $f_\alpha$ and $g_\alpha$ with domain $\delta$ such that $[f_\alpha]_U=\gamma_\alpha$ and $[g_\alpha]_U = A\cap\gamma_\alpha$ (recall that we are identifying $\Ult{\VV}{U}$ with its transitive collapse). 
 Let $\map{c}{[\kappa]^2}{U}$ denote the unique function with the property that $$c(\{\alpha,\beta\}) ~ = ~ \Set{\xi<\delta}{f_\alpha(\xi)<f_\beta(\xi), ~ g_\alpha(\xi)=f_\alpha(\xi)\cap g_\beta(\xi)}$$ holds for all $\alpha<\beta<\kappa$. 
  In this situation, since $\kappa>\delta$ is weakly compact, we know that $\betrag{U}=2^\delta<\kappa$ and hence the weak compactness of $\kappa$ yields 
  %\todo{S.: Why? Can we write this as a function into an ordinal $\beta < \kappa$? Maybe add a reference. 
  %
%  P.: Wegen des ersten Teils kann man annehmen, dass $\kappa>2^\delta=\betrag{U}$. Ich habe was ergaenzt.} 
 an unbounded subset $H$ of $\kappa$ and an element $X$ of $U$ with the property that $c[H]^2=\{X\}$. Pick a function $g$ with domain $\delta$ and the property that $g(\xi)=\bigcup\Set{g_\alpha(\xi)}{\alpha\in H}$ holds for all $\xi\in X$. 
  This construction ensures that $[g]_U\cap\gamma_\alpha=[g_\alpha]_U$ holds for every $\alpha\in H$ and we can conclude that $[g]_U=A$. In particular, the set $A$ is not fresh over $\Ult{\VV}{U}$.

  (ii) Fix a ${<}(2^\delta)^+$-closed ultrafilter $F$ on $\cof{\lambda}$ that contains all cobounded subsets of $\cof{\lambda}$ and assume, towards a contradiction, that $A$ is an unbounded subset of $\lambda$ that is fresh over $\Ult{\VV}{U}$. 
  Pick a strictly increasing sequence $\seq{\gamma_\eta}{\eta<\cof{\lambda}}$ that is cofinal in $\lambda$. Given $\eta<\cof{\lambda}$, fix functions $f_\eta$ and $g_\eta$  with domain $\delta$ that satisfy $[f_\eta]_U=\gamma_\eta$ and $[g_\eta]_U=A\cap\gamma_\eta$. Moreover, given $\eta<\cof{\lambda}$ and $X\in U$, we set $$A_{\eta,X} ~ = ~ \Set{\zeta\in(\eta,\cof{\lambda})}{X=\Set{\xi<\delta}{f_\eta(\xi)<f_\zeta(\xi), ~ g_\eta(\xi)=g_\zeta(\xi)\cap f_\eta(\xi)}}.$$
  Then $$\bigcup\Set{A_{\eta,X}}{X\in U} ~ = ~ (\eta,\cof{\lambda})$$ holds for every $\eta<\cof{\lambda}$. By our assumptions on $F$, there exists a sequence $\seq{X_\eta}{\eta<\cof{\lambda}}$ of elements of $U$ with the property that $A_{\eta,X_\eta}\in F$ holds for all $\eta<\cof{\lambda}$. Furthermore, as $\cof{\lambda} > 2^\delta$, there is a set $X_*\in U$ and an unbounded subset $E$ of $\cof{\lambda}$ with $X_\eta=X_*$ for all $\eta\in E$. By construction, we now have $f_\eta(\xi)<f_\zeta(\xi)$ and $g_\eta(\xi)=g_\zeta(\xi)\cap f_\eta(\xi)$ for all $\xi\in X_*$ and all $\eta,\zeta\in E$ with $\eta<\zeta$. Pick a function $g$ with domain $\delta$ and $$g(\xi) ~ = ~ \bigcup\Set{g_\eta(\xi)}{\eta\in E}$$ for all $\xi\in X_*$. Then $[g]_U\cap\gamma_\eta=[g_\eta]_U=A\cap\gamma_\eta$ for all $\eta\in E$ and hence we can conclude that $[g]_U = A\in\Ult{\VV}{U}$, a contradiction. The second part of the statement follows directly from the first part and the \emph{filter extension property} of strongly compact cardinals (see {\cite[Proposition 4.1]{MR1994835}}).  
\end{proof}

\begin{proof}[Proof of Proposition \ref{proposition:MinFresh}]
 (i) If $\kappa>\delta$ is the minimal cardinal with $\POT{\kappa}\nsubseteq\Ult{\VV}{U}$, then every element of  $\POT{\kappa}\setminus\Ult{\VV}{U}$ is unbounded in $\kappa$ and fresh over $\Ult{\VV}{U}$. 
 
 (ii) Fix a limit ordinal $\lambda$ with $\cof{\lambda}^{\Ult{\VV}{U}}=j_U(\delta^+)$ and pick a strictly increasing, cofinal function $\map{c}{j_U(\delta^+)}{\lambda}$ in $\Ult{\VV}{U}$. 
  Since $j_U[\delta^+]$ is a cofinal subset of $j_U(\delta^+)$ of order-type $\delta^+$, we know that $(c\circ j_U)[\delta^+]$ is a cofinal subset of $\lambda$ of order-type $\delta^+$. In particular, the closure of $\Ult{\VV}{U}$ under $\delta$-sequences implies that every proper initial segment of $(c\circ j_U)[\delta^+]$ is an element of $\Ult{\VV}{U}$. Finally, since {\cite[Proposition 22.4]{MR1994835}} shows that $j_U[\delta^+]\notin\Ult{\VV}{U}$, we can conclude that the set $(c\circ j_U)[\delta^+]$ is fresh over $\Ult{\VV}{U}$.  
 
  (iii) First, assume that $\cof{\lambda}^{\Ult{\VV}{U}}=\delta^+$. Since $2^\delta = \delta^+$ and $U \notin \Ult{\VV}{U}$, we have $\POT{\delta^+}\nsubseteq\Ult{\VV}{U}$, and we can use (i) to find an unbounded subset $A$ of $\delta^+$ that is fresh over $\Ult{\VV}{U}$. Let $\seq{\gamma_\alpha}{\alpha<\delta^+}$ be the monotone enumeration of an unbounded subset of $\lambda$ of order-type $\delta^+$ in $\Ult{\VV}{U}$. Set $B=\Set{\gamma_\alpha}{\alpha\in A}$. Then $B$ is unbounded in $\lambda$ and it is easy to see that $B$ is fresh over $\Ult{\VV}{U}$. 
 
 Now, assume that $\cof{\lambda}^{\Ult{\VV}{U}}>\delta^+$ and fix an unbounded subset $A$ of $\lambda$ of order-type $\delta^+$. Then the closure of $\Ult{\VV}{U}$ under $\delta$-sequences implies that $A$ is fresh over $\Ult{\VV}{U}$. 
\end{proof}

Note that, in the situation of Proposition \ref{proposition:MinFresh}, we have $\cof{\lambda}=\delta^+$ for every limit ordinal $\lambda$ with $\cof{\lambda}^{\Ult{\VV}{U}}=j_U(\delta^+)$. 
 In particular, if $\kappa$ is a strong limit cardinal of cofinality $\delta^+$, then the fact that $j_U(\kappa)=\kappa$ holds allows us to use the second part of the above proposition to conclude that there is an unbounded subset of $\kappa$ that is fresh over $\Ult{\VV}{U}$.  
 Moreover, the results of Cummings in \cite{MR1217188} discussed in the first section show that the cardinal arithmetic assumption in the third part of the proposition can, in general, not be omitted.

%%%%%%%%%%%%%%%%%%%%%%%%%%%%%%

\section{Fresh subsets at image points of ultrapower embeddings}

In this section, we will use a modified square principle introduced in \cite{MR3694332} to show that the existence of a $\square(\kappa)$-sequence allows us to construct a fresh subset of $j_U(\kappa)$. 
 The principle defined in the next definition is a variation of the indexed square principles studied in \cite{MR1838355} and \cite{MR1942302}. 
%Using results from \cite{MR3893286}, we \ldots 

\begin{definition}[Lambie-Hanson]\label{ind_square_def}
	Let $\delta < \kappa$ be infinite regular cardinals. A $\square^{\mathrm{ind}}(\kappa, \delta)$-sequence is a matrix $$\seq{C_{\gamma,\xi}}{\gamma < \kappa, ~ i(\gamma) \leq \xi < \delta}$$  satisfying the following statements: 
  \begin{enumerate}[ref=(\roman{enumi})]
		\item If $\gamma \in \Lim\cap\kappa$, then $i(\gamma) < \delta$. 

		\item If $\gamma \in\Lim\cap\kappa$ and $i(\gamma) \leq \xi < \delta$, then $C_{\gamma,\xi}$ is a closed unbounded subset of  $\gamma$. 

		\item If $\gamma\in\Lim\cap\kappa$ and $i(\gamma) \leq \xi_0 < \xi_1 < \delta$, then $C_{\gamma,\xi_0} \subseteq C_{\gamma,\xi_1}$. 

		\item If $\beta,\gamma\in\Lim\cap\kappa$ and $i(\gamma) \leq \xi < \delta$ with $\beta \in \Lim(C_{\gamma,\xi})$, then  $\xi\geq i(\beta)$ and $ C_{\beta,\xi}=C_{\gamma,\xi} \cap \beta$. 

		\item If $\beta,\gamma\in\Lim\cap\kappa$ with $\beta < \gamma$, then there is an $i(\gamma)\leq \xi < \delta$ such that $\beta \in \Lim(C_{\gamma,\xi})$.

    \item \label{it:(vi)} There is no closed unbounded subset  $C$ of $\kappa$ with the property that, for all $\gamma \in \Lim(C)$, there is $\xi < \delta$ such that $C_{\gamma,\xi}=C \cap \gamma$ holds.
	\end{enumerate}
%	We let $\square^{\mathrm{ind}}(\kappa, \delta)$ denote the assertion that there is a $\square^{\mathrm{ind}}(\kappa, \delta)$-sequence. 
\end{definition}

 The main result of \cite{MR3893286} now shows that for all infinite regular cardinals $\delta<\kappa$, the existence of a $\square(\kappa)$-sequence implies the existence of a $\square^{\mathrm{ind}}(\kappa, \delta)$-sequence. 
 The proof of the following result is based on this implication.

\begin{theorem}\label{theorem:IndSquareFresh}
  Let $U$ be a normal ultrafilter on a measurable cardinal $\delta$ and let $\kappa>\delta$ be a regular cardinal. 
  %\todo{Isn't this true for all regular cardinals above $\delta$?  with $j_U(\kappa)=\sup(j_U[\kappa])$} 
  If  there exists a $\square(\kappa)$-sequence, then there is a closed unbounded subset of $j_U(\kappa)$ that is fresh over $\Ult{\VV}{U}$. 
\end{theorem}

\begin{proof}
 By {\cite[Theorem 3.4]{MR3893286}}, our assumptions allow us to fix a $\square^{\mathrm{ind}}(\kappa, \delta)$-sequence $$\seq{C_{\gamma,\xi}}{\gamma < \kappa, ~ i(\gamma) \leq \xi < \delta}.$$  
  Given $\gamma\in\Lim\cap\kappa$, let $\map{f_\gamma}{\delta}{\POT{\gamma}}$ denote the unique function with $f_{\gamma}(\xi)=\emptyset$ for all $\xi<i(\gamma)$ and $f_\gamma(\xi)=C_{\gamma,\xi}$ for all $i(\gamma)\leq\xi<\delta$. 
  In this situation, {\L}os' Theorem directly implies that for all $\beta,\gamma\in\Lim\cap\kappa$ with $\beta\leq\gamma$, the set $[f_\gamma]_U$ is closed unbounded in $j_U(\gamma)$ and $[f_\beta]_U=[f_\gamma]_U\cap j_U(\beta)$ holds. 
  Define $$A ~ = ~ \bigcup\Set{[f_\gamma]_U}{\gamma \in \Lim \cap \kappa}.$$ By our assumptions on $\kappa$, we know that $j_U(\kappa)=\sup(j_U[\kappa])$ and therefore $A$ is a closed unbounded subset of $j_U(\kappa)$ with $A\cap j_U(\gamma)=[f_\gamma]_U$ for all $\gamma\in \Lim \cap \kappa$. 
 
 Assume, towards a contradiction, that $A$ is an element of $\Ult{\VV}{U}$. Then there is $\map{f}{\delta}{\POT{\kappa}}$ with $[f]_U=A$ and the property that $f(\xi)$ is a closed unbounded subset of $\kappa$ for all $\xi<\delta$. 
 Since we have $\Set{\xi<\delta}{f_\gamma(\xi)=f(\gamma)\cap\gamma\neq\emptyset}\in U$ for all $\gamma \in \Lim \cap \kappa$, we can find $\xi<\delta$ with the property that $\xi\geq i(\gamma)$ and $f(\xi)\cap\gamma=C_{\gamma,\xi}$ holds for unboundedly many $\gamma$ below $\kappa$. 
  Pick $\beta\in\Lim(f(\xi))$ and $\beta<\gamma<\kappa$ with $\xi\geq i(\gamma)$ and $f(\xi)\cap\gamma=C_{\gamma,\xi}$. Then $\beta\in\Lim(C_{\gamma,\xi})$ and this implies that $\xi\geq i(\beta)$ and $C_{\beta,\xi}=C_{\gamma,\xi}\cap\beta=f(\xi)\cap\beta$. 
 These computations show that there is a closed unbounded subset $C$ of $\kappa$ and $\xi<\delta$ such that $\xi\geq i(\beta)$ and $C\cap\beta=C_{\beta,\xi}$ holds for all $\beta\in\Lim(C)$. But this contradicts \ref{it:(vi)} in Definition \ref{ind_square_def}. %{\cite[Proposition 3.3]{MR3893286}}. 
\end{proof}

%%%%%%%%%%%%%%%%%%%%%%%%%%%%%%

\section{Fresh subsets of successors of singular cardinals}

We now aim to construct fresh subsets of cardinals that are not contained in the image of the corresponding ultrapower embedding, e.g. successors of singular cardinals whose cofinality is equal to the relevant measurable cardinal. 
Our arguments will rely on two standard observations about measurable ultrapowers and $\square_\kappa$-sequences that we present first. 
A proof of the following lemma is contained in the proof of {\cite[Lemma 1.3]{MR3845129}}.

\begin{lemma}\label{lemma:UltrapowerEmbeddingFixedPoints}
 Let $U$ be a normal ultrafilter on a measurable cardinal $\delta$. If $\nu>\delta$ is a cardinal with $\cof{\nu}\neq\delta$ and $\lambda^\delta<\nu$ for all $\lambda<\nu$, then $j_U(\nu)=\nu$ and $j_U(\nu^+)=\nu^+$. 
\end{lemma}

The next lemma contains a well-known construction (see {\cite[Section 4]{zbMATH05011391}}) that shows that, in the situations relevant for our proofs, the existence of some $\square_\kappa$-sequence already implies the existence of such a sequence with certain additional structural properties.

\begin{lemma}\label{lemma:ModSquare}
 Let $\kappa$ be a singular cardinal and let $S$ be a stationary subset of $\kappa^+$. If there exists a $\square_\kappa$-sequence, then there exists a $\square_\kappa$-sequence $\seq{C_\gamma}{\gamma\in\Lim\cap\kappa^+}$ and a stationary subset $E$ of $S$ such that $\otp{C_\gamma}<\kappa$ and $\Lim(C_\gamma)\cap E=\emptyset$ for all $\gamma\in\Lim\cap\kappa^+$. 
\end{lemma}

\begin{proof}
 Fix a $\square_\kappa$-sequence $\seq{A_\gamma}{\gamma\in\Lim\cap\kappa^+}$ and a closed unbounded subset $C$ of $\kappa$ of order-type $\cof{\kappa}$. 
    Given $\gamma\in\Lim\cap\kappa^+$, let $\lambda_\gamma=\otp{A_\gamma}\leq\kappa$ and let $\seq{\beta^\gamma_\alpha}{\alpha<\lambda_\gamma}$ denote the monotone enumeration of $A_\gamma$. 
  Given $\gamma\in\Lim\cap\kappa^+$ with $\lambda_\gamma\in\Lim(C)\cup\{\kappa\}$, let $B_\gamma=\Set{\beta\in A_\gamma}{\otp{A_\gamma\cap\beta}\in C}$. 
 Next, if $\gamma\in\Lim\cap\kappa^+$ with $\lambda_\gamma\notin\Lim(C)\cup\{\kappa\}$ and $\Lim(C)\cap\lambda_\gamma=\emptyset$, then we define $B_\gamma=A_\gamma$. 
  Finally, if $\gamma\in\Lim\cap\kappa^+$ with $\lambda_\gamma\notin\Lim(C)\cup\{\kappa\}$ and $\Lim(C)\cap\lambda_\gamma\neq\emptyset$, then we set $\alpha=\max(\Lim(C)\cap\lambda_\gamma)<\lambda_\gamma$ and we define $B_\gamma=B_{\beta^\gamma_\alpha}\cup(A_\gamma\setminus B_{\beta^\gamma_\alpha})$. 
 
 \begin{claim*}
  The sequence $\seq{B_\gamma}{\gamma\in\Lim\cap\kappa^+}$ is a $\square_\kappa$-sequence with $\otp{B_\gamma}<\kappa$ for all $\gamma\in\Lim\cap\kappa^+$. \qed
 \end{claim*}
 
  With the help of Fodor's Lemma, we can now find a stationary subset $E$ of $S$ and $\lambda<\kappa$ with $\otp{B_\gamma}=\lambda$ for all $\gamma\in E$. Then we have $\betrag{\Lim(B_\gamma)\cap E}\leq 1$ for all $\gamma\in\Lim\cap\kappa^+$. 
  Given $\gamma\in\Lim\cap\kappa^+$, define $C_\gamma=B_\gamma$ if $\otp{B_\gamma}\leq\lambda$ and let $C_\gamma=\Set{\beta\in B_\gamma}{\otp{B_\gamma\cap\beta}>\lambda}$ if $\otp{B_\gamma}>\lambda$. 
  
   \begin{claim*}
  The sequence $\seq{C_\gamma}{\gamma\in\Lim\cap\kappa^+}$ is a $\square_\kappa$-sequence with $\otp{C_\gamma}<\kappa$ and $\Lim(C_\gamma)\cap E=\emptyset$ for all $\gamma\in\Lim\cap\kappa^+$. \qed
 \end{claim*}
 
 This completes the proof of the lemma. 
\end{proof}

 We are now ready to prove the main result of this section that will allow us to handle successors of singular cardinals of measurable cofinality in the proof of Theorem \ref{theorem:CanonicalModelsCharFresh}.

\begin{theorem}\label{theorem:SuccSingularFresh}
 Let $U$ be a normal ultrafilter on a measurable cardinal $\delta$ and let $\kappa$ be a singular  cardinal of cofinality $\delta$ with $2^\kappa=\kappa^+$ and the property that $\lambda^\delta<\kappa$ holds for all $\lambda<\kappa$. 
 If there exists a $\square_\kappa$-sequence, then there is a closed unbounded subset of $\kappa^+$ that is fresh over $\Ult{\VV}{U}$. 
\end{theorem}

\begin{proof}
  By our assumptions, we can apply Lemma \ref{lemma:ModSquare} to obtain a $\square_\kappa$-sequence $\seq{C_\gamma}{\gamma\in\Lim\cap\kappa^+}$ and a stationary subset $E$ of $S^{\kappa^+}_\delta$ such that $\otp{C_\gamma}<\kappa$ and $\Lim(C_\gamma)\cap E=\emptyset$ for all $\gamma\in\Lim\cap\kappa^+$. 
  Next, note that Lemma \ref{lemma:UltrapowerEmbeddingFixedPoints} implies that $j_U((\nu^\delta)^+)=(\nu^\delta)^+<\kappa$ holds for all cardinals $\nu<\kappa$. This allows us to fix the monotone enumeration $\seq{\kappa_\xi}{\xi<\delta}$ of a closed unbounded subset of $\kappa$ of order-type $\delta$ with the property that $j_U(\kappa_\xi)=\kappa_\xi$ holds for all $\xi<\delta$. 
  In this situation, the normality of $U$ implies that $[\xi\mapsto\kappa_\xi]_U=\kappa$ and $[\xi\mapsto\kappa_\xi^+]_U\leq\kappa^+$. 
  Given $\gamma\in\Lim\cap\kappa^+$, let $\xi_\gamma$ denote the minimal element $\xi$ of $\delta$ with $\kappa_\xi^+>\otp{C_\gamma}$. Note that  $\xi_\gamma\geq\xi_\beta$ holds for all $\gamma\in\Lim\cap\kappa^+$ and $\beta\in\Lim(C_\gamma)$.  
 
  In the following, we inductively construct a sequence $$\seq{f_\gamma\in\prod_{\xi<\delta}\kappa_\xi^+}{\gamma<\kappa^+}.$$
  The idea behind this construction is that these functions represent 
  %the ordinals below 
  a cofinal subset of $\kappa^+$ 
  %in $\Ult{\VV}{U}$ 
  and thereby in particular witness that $[\xi\mapsto\kappa_\xi^+]_U = \kappa^+$. 
  We identify each $f_\gamma \in \prod_{\xi<\delta}\kappa_\xi^+$ with a function with domain $\delta$ in the obvious way and define: 
  \begin{itemize}
   \item $f_0(\xi)=0$ for all $\xi<\delta$. 
   
   \item $f_{\gamma+1}(\xi)=f_\gamma(\xi)+1$ for all $\gamma<\kappa^+$ and $\xi<\delta$. 
   
   \item If $\gamma\in\Lim\cap\kappa^+$ with $\Lim(C_\gamma)$ bounded in $\gamma$ and $\xi<\delta$, then 
    %$$f_\gamma(\xi) ~ = ~ \sup\Set{f_\beta(\xi)}{\beta\in C_\gamma\setminus\max(\Lim(C_\gamma))}.$$
     $$f_\gamma(\xi) ~ = ~ \min\Set{\lambda\in\Lim}{\textit{$\lambda>f_\beta(\xi)$ for all $\beta\in C_\gamma\setminus\max(\Lim(C_\gamma))$}}.$$ 
   
   \item If $\gamma\in\Lim\cap\kappa^+$ with $\Lim(C_\gamma)$ unbounded in $\gamma$ and $\xi<\xi_\gamma$, then $f_\gamma(\xi)=\omega$. 
   
      \item If $\gamma\in\Lim\cap\kappa^+$ with $\Lim(C_\gamma)$ unbounded in $\gamma$ and $\xi_\gamma\leq\xi<\delta$, then $$f_\gamma(\xi) ~ = ~ \sup\Set{f_\beta(\xi)}{\beta\in\Lim(C_\gamma)}.$$ 
  \end{itemize}

  \begin{claim*}
   \begin{enumerate}
    \item If $\beta<\gamma<\kappa^+$, then $f_\beta(\xi)<f_\gamma(\xi)$ for coboundedly many $\xi<\delta$. 
    
        \item  If $\gamma\in\Lim\cap\kappa^+$, $\beta\in\Lim(C_\gamma)$ and $\xi_\gamma\leq\xi<\delta$, then $f_\beta(\xi)<f_\gamma(\xi)$. 
        
    \item If $\gamma\in\Lim\cap\kappa^+$, then $f_\gamma(\xi)\in\Lim$ for all $\xi<\delta$. 
   \end{enumerate}
  \end{claim*}
  
  \begin{proof}[Proof of the Claim]
   (i) We prove the statement by induction on $0<\gamma<\kappa^+$, where the successor step follows trivially from our induction hypothesis. 
   Now, assume that $\gamma\in\Lim\cap\kappa^+$ with $\Lim(C_\gamma)$ bounded in $\gamma$. Since $\delta$ is an uncountable regular cardinal, our induction hypothesis allows us to find $\zeta<\delta$ with the property that $f_{\beta_0}(\xi)<f_{\beta_1}(\xi)$ holds for all $\beta_0,\beta_1\in C_\gamma\setminus\max(\Lim(C_\gamma))$ with $\beta_0<\beta_1$ and all $\zeta\leq\xi<\delta$. By definition, we now have $$f_\gamma(\xi) ~ = ~ \sup\Set{f_\beta(\xi)}{\beta\in C_\gamma\setminus\max(\Lim(C_\gamma))}$$ for all $\zeta\leq\xi<\delta$.  Since $C_\gamma\setminus\max(\Lim(C_\gamma))$ is a cofinal subset of $\gamma$, the desired statement for $\gamma$ now follows directly from our induction hypothesis. 
  Finally, if $\gamma\in\Lim\cap\kappa^+$ with $\Lim(C_\gamma)$ unbounded in $\gamma$, then the desired statement for $\gamma$ follows directly from the definition of $f_\gamma$ and our induction hypothesis. 
   
   (ii) We  prove the claim by induction on $\gamma\in\Lim\cap\kappa^+$. 
    First, if $\gamma\in\Lim\cap\kappa^+$ with $\Lim(C_\gamma)$ bounded in $\gamma$ and $\beta=\max(\Lim(C_\gamma))$, then our definition ensures that $f_\beta(\xi)<f_\gamma(\xi)$ holds for all $\xi<\delta$ and hence the desired statement follows directly from our induction hypothesis.  
    Next, if $\gamma\in\Lim\cap\kappa^+$ with $\Lim(C_\gamma)$ unbounded in $\gamma$, then our induction hypothesis implies that $f_{\beta_0}(\xi)<f_{\beta_1}(\xi)$ holds for all $\beta_0,\beta_1\in\Lim(C_\gamma)$ with $\beta_0<\beta_1$ and all $\xi_{\beta_1}\leq\xi<\delta$. 
    Since $\xi_\gamma\geq\xi_\beta$ holds for all $\beta\in\Lim(C_\gamma)$, this fact together with our definition yields the desired statement for $\gamma$.     
    
   (iii) This statement is a direct consequence of the definition of the sequence $\seq{f_\gamma}{\gamma < \kappa^+}$ and statement (ii). 
  \end{proof}

 Note that the first part of the above claim in particular shows that we have $[f_\beta]_U<[f_\gamma]_U<[\xi\mapsto\kappa_\xi^+]_U$  for all $\beta<\gamma<\kappa^+$. 
   Since we already observed that $[\xi\mapsto\kappa_\xi^+]_U=(\kappa^+)^{\Ult{\VV}{U}}\leq\kappa^+$ holds,  we can conclude that $[\xi\mapsto\kappa_\xi^+]_U=\kappa^+$.

   Next, notice that the fact that $2^\kappa=\kappa^+$ holds allows us to fix an enumeration $\seq{h_\alpha}{\alpha<\kappa^+}$ of $\prod_{\xi<\delta}\POT{\kappa_\xi^+}$ of order-type $\kappa^+$. 
 In addition, let $\seq{\gamma_\alpha}{\alpha<\kappa^+}$ denote the monotone enumeration of $E$. 
%
% Fix a $\Diamond(E)$-sequence $\seq{D_\gamma}{\gamma\in E}$. 
 %
% Given $\gamma\in E$ and $\xi_\gamma\leq\xi<\delta$, define $$C^\gamma_\xi ~ = ~ \bigcup\Set{h_\alpha(\xi)\cap f_\gamma(\xi)}{\alpha\in D_\gamma} ~ \subseteq ~ f_\gamma(\xi).$$  
  %
%  Let $B$ denote the set of all $\gamma\in E\cap\Lim(A)$ with the property that the set $C^\gamma_\xi$ is closed unbounded in $f_\gamma(\xi)$ for all $\xi_\gamma\leq\xi<\delta$.  
  %
%  Note that, since $B\subseteq S^{\kappa^+}_\delta$, the set $\Lim(C^\gamma_\xi)$ is closed unbounded in $\gamma$ for every $\gamma\in B$ and all $\xi_\gamma\leq\xi<\delta$. 
  %
  We now inductively define a sequence $$\seq{c_\gamma}{\gamma\in\Lim\cap\kappa^+}$$ of functions with domain $\delta$   satisfying the following statements  for all $\gamma\in\Lim\cap\kappa^+$: 
  \begin{enumerate}[label = (\alph{enumi}), ref = (\alph{enumi})]
   \item $c_\gamma(\xi)$ is a closed unbounded subset of $f_\gamma(\xi)$ for all $\xi<\delta$. 
   
%   \item[(b)] Given $\xi<\delta$ with $\kappa_\xi^+>\otp{C_\gamma}$, the set $c_\gamma(\xi)$ is unbounded in $f_\gamma(\xi)$ if and only if $\gamma\notin E$. 
   
   \item If $\beta\in\Lim\cap\gamma$, then $f_\beta(\xi)<f_\gamma(\xi)$ and $c_\beta(\xi)=c_\gamma(\xi)\cap f_\beta(\xi)$ for coboundedly many $\xi<\delta$. 
   
   \item If $\gamma\notin E$, then $c_\beta(\xi)=c_\gamma(\xi)\cap f_\beta(\xi)$ for all $\beta\in\Lim(C_\gamma)$ and $\xi_\gamma\leq\xi<\delta$. 
   
   \item\label{it:(d)} If $\gamma \in E$ and $\alpha<\kappa^+$ with $\gamma=\gamma_\alpha$, then $c_\gamma(\xi)\neq h_\alpha(\xi)\cap f_\gamma(\xi)$ for all $\xi_\gamma\leq\xi<\delta$. 
  \end{enumerate}

  The idea behind this definition is to use the fact that the sequence $\seq{[f_\gamma]_U}{\gamma<\kappa^+}$ is not continuous at ordinals of cofinality $\delta$ to \emph{diagonalize} against the sequence $\seq{[h_\alpha]_U}{\alpha<\kappa^+}$ of subsets of $\kappa^+$ in $\Ult{\VV}{U}$ in \ref{it:(d)}. 
  The inductive definition of this sequence is straightforward, but we decided to give the details to convince the reader that it works. We distinguish between the following cases: 
  
   \setcounter{case}{0}

   \begin{case}
     $\gamma\in\Lim\cap\kappa^+$ with $\Lim\cap\gamma$ bounded in $\gamma$.
   \end{case}
   
   First, we set $\beta_0=0$ if $\Lim(C_\gamma)=\emptyset$ and $\beta_0=\max(\Lim(C_\gamma))$ otherwise. 
   Next, we set $\beta_1=0$ if $\gamma=\omega$ and $\beta_1=\max(\Lim\cap\gamma)$ otherwise. 
  We than have $\beta_0\leq\beta_1<\gamma$ and $f_{\beta_0}(\xi)<f_\gamma(\xi)$ for all $\xi<\delta$. 
  Using our induction hypothesis, we can find $\xi_\gamma\leq\zeta<\delta$ with the property that $f_{\beta_0}(\xi)\leq f_{\beta_1}(\xi)<f_\gamma(\xi)$ holds for all $\zeta\leq\xi<\delta$ and, if $\beta_0>0$, then $c_{\beta_0}(\xi)=c_{\beta_1}(\xi)\cap f_{\beta_0}(\xi)$ for all $\zeta\leq\xi<\delta$. 
      Note that our assumptions imply that $\Lim(C_\gamma)$ is bounded in $\gamma$ and hence the definition of $f_\gamma(\xi)$ ensures that $\cof{f_\gamma(\xi)}=\omega$  holds for every $\xi<\delta$. 
      Therefore, we can fix a sequence of strictly increasing functions $\seq{\map{k_\xi}{\omega}{f_\gamma(\xi)}}{\xi<\delta}$ with the property that $k_\xi$ is cofinal in $f_\gamma(\xi)$ for all $\xi<\delta$, $k_\xi(0)=f_{\beta_0}(\xi)$ for all $\xi<\zeta$, and $k_\xi(0)=f_{\beta_1}(\xi)$ for all $\zeta\leq\xi<\delta$. 
      Define 
 \begin{equation*}
    c_\gamma(\xi) ~ = ~  
   \begin{cases}
     %\Set{k_\xi(n)}{n<\omega}, & \text{ for all } \xi<\xi_\gamma\\
     \Set{k_\xi(n)}{n<\omega}, & \textit{for all $\xi<\zeta$, if $\beta_0=0$.} \\
    c_{\beta_0}(\xi) ~ \cup ~ \Set{k_\xi(n)}{n<\omega}, & \textit{for all $\xi<\zeta$,  if $\beta_0>0$.} \\
   \Set{k_\xi(n)}{n<\omega}, & \textit{for all $\zeta\leq\xi<\delta$, if $\beta_1=0$.} \\ 
   c_{\beta_1}(\xi) ~ \cup ~ \Set{k_\xi(n)}{n<\omega}, & \textit{for all $\zeta\leq\xi<\delta$, if $\beta_1>0$.}
  \end{cases}
 \end{equation*}
  
    These definitions ensure that $c_\gamma(\xi)$ is a closed unbounded subset of $f_\gamma(\xi)$ for all $\xi<\delta$ . Moreover, if $\beta_0>0$, then $c_\gamma(\xi)\cap f_{\beta_0}(\xi)=c_{\beta_0}(\xi)$ holds for all $\xi<\delta$. This inductively implies that $c_\gamma(\xi)\cap f_\beta(\xi)=c_\beta(\xi)$ holds for all $\beta\in\Lim(C_\gamma)$ and all $\xi_\gamma\leq\xi<\delta$. 
  Next, if $\beta_1>0$ and $\zeta\leq\xi<\delta$, then $f_{\beta_1}(\xi)<f_\gamma(\xi)$ and $c_\gamma(\xi)\cap f_{\beta_1}(\xi)=c_{\beta_1}(\xi)$. 
   This allows us to conclude that for all $\beta\in\Lim\cap\gamma$, we have $f_\beta(\xi)<f_\gamma(\xi)$ and $c_\beta(\xi)=c_\gamma(\xi)\cap f_\beta(\xi)$ for coboundedly many $\xi<\delta$.

  \begin{case}
 $\gamma\in\Lim\cap\kappa^+$ with $\Lim\cap\gamma$ unbounded in $\gamma$ and $\Lim(C_\gamma)$ bounded in $\gamma$.
\end{case}

 Since our assumptions imply that $\cof{\gamma}=\omega$, there is a strictly increasing sequence $\seq{\beta_n}{n<\omega}$ cofinal in $\gamma$ such that $\beta_n\in\Lim\cap\gamma$ for all $0<n<\omega$, $\beta_0=0$ in case $\Lim(C_\gamma)=\emptyset$, and  $\beta_0=\max(\Lim(C_\gamma))$ in case $\Lim(C_\gamma)\neq\emptyset$. 
  By the regularity of $\delta$, we can find $\xi_\gamma\leq\zeta<\delta$ such that $f_{\beta_n}(\xi)<f_{\beta_{n+1}}(\xi)<f_\gamma(\xi)$ and $c_{\beta_{n+2}}(\xi)\cap f_{\beta_{n+1}}(\xi)=c_{\beta_{n+1}}(\xi)$ for all $\zeta\leq\xi<\delta$ and all $n<\omega$ and, if  $\beta_0>0$, then $c_{\beta_1}(\xi)\cap f_{\beta_0}(\xi)=c_{\beta_0}(\xi)$ for all $\zeta\leq\xi<\delta$. 
  By the definition of $f_\gamma$, we then have  $f_\gamma(\xi)=\sup\Set{f_{\beta_n}(\xi)}{n<\omega}$ for all $\zeta\leq\xi<\delta$. 
  Since the definition of $f_\gamma$ also implies that $\cof{f_\gamma(\xi)}=\omega$  and $f_{\beta_0}(\xi)<f_\gamma(\xi)$ for all $\xi<\delta$, we can fix a sequence of strictly increasing functions $\seq{\map{k_\xi}{\omega}{f_\gamma(\xi)}}{\xi<\zeta}$ with the property that $k_\xi$ is cofinal in $f_\gamma(\xi)$ for all $\xi<\zeta$ and $k_\xi(0)=f_{\beta_0}(\xi)$ for all $\xi<\zeta$.
  Define  
  \begin{equation*}
    c_\gamma(\xi) ~ = ~ \begin{cases}
    %\Set{k_\xi(n)}{n<\omega}, & \text{ for all } \xi<\xi_\gamma\\
   \Set{k_\xi(n)}{n<\omega}, & \textit{for all $\xi<\zeta$, if $\beta_0=0$.}\\
    c_{\beta_0}(\xi) ~ \cup ~ \Set{k_\xi(n)}{n<\omega}, & \textit{for all $\xi<\zeta$, if $\beta_0>0$.}\\
  \bigcup\Set{c_{\beta_n}(\xi)}{0<n<\omega}, & \textit{for all $\zeta\leq\xi<\delta$,  if $\beta_0=0$.}\\
  \bigcup\Set{c_{\beta_n}(\xi)}{n<\omega}, & \textit{for all $\zeta\leq\xi<\delta$,  if $\beta_0>0$.}
  \end{cases}
 \end{equation*}
  
  Then the set $c_\gamma(\xi)$ is closed and unbounded in $f_\gamma(\xi)$ for all $\xi<\delta$. 
  In addition, if $\beta_0>0$, then $c_\gamma(\xi)\cap f_{\beta_0}(\xi)=c_{\beta_0}(\xi)$ for all $\xi<\delta$. 
   In particular, we have $c_\gamma(\xi)\cap f_\beta(\xi)=c_\beta(\xi)$ for all $\beta\in\Lim(C_\gamma)$ and all $\xi_\gamma\leq\xi<\delta$. 
  Next, if $0<n<\omega$ and $\zeta\leq\xi<\delta$, then $f_{\beta_n}(\xi)<f_\gamma(\xi)$ and $c_\gamma(\xi)\cap f_{\beta_n}(\xi)=c_{\beta_n}(\xi)$. This directly implies that for all $\beta\in\Lim\cap\gamma$, we have $f_\beta(\xi)<f_\gamma(\xi)$ and $c_\gamma(\xi)\cap f_\beta(\xi)=c_\beta(\xi)$ for coboundedly many $\xi<\delta$. 
  
  \begin{case}
    $\gamma\in\Lim\cap\kappa^+$ with $\gamma\notin E$ and $\Lim(C_\gamma)$ unbounded in $\gamma$.
  \end{case}

Let 
 \begin{equation*}
   c_\gamma(\xi) ~ = ~ \begin{cases}
  \omega, & \textit{for all $\xi<\xi_\gamma$.}\\ 
   \bigcup\Set{c_\beta(\xi)}{\beta\in\Lim(C_\gamma)}, & \textit{for all $\xi_\gamma\leq\xi<\delta$.}
  \end{cases}
 \end{equation*}
  
 %  Set $c_\gamma(\xi)=\omega$ for all $\xi<\xi_\gamma$ and $c_\gamma(\xi)=\bigcup\Set{c_\beta(\xi)}{\beta\in\Lim(C_\gamma)}$ for all $\xi_\gamma\leq\xi<\delta$. 
   %
   Fix $\beta_0,\beta_1\in\Lim(C_\gamma)$ with $\beta_0<\beta_1$. 
   Then $\beta_0\in\Lim(C_{\beta_1})$ and $\beta_1\notin E$ by the choice of the $\square_\kappa$-sequence and the stationary set $E$, because we have $\beta_1 \in \Lim(C_\gamma)$. 
   Moreover, if $\xi_\gamma\leq\xi<\delta$, 
  then the set $c_{\beta_1}(\xi)$ is unbounded in $f_{\beta_1}(\xi)$, $\xi\geq\xi_{\beta_1}$, $f_{\beta_0}(\xi)<f_{\beta_1}(\xi)$ and $c_{\beta_0}(\xi)=c_{\beta_1}(\xi)\cap f_{\beta_0}(\xi)$. 
   This shows that $c_\gamma(\xi)$ is a closed unbounded subset of $f_\gamma(\xi)$ for all $\xi<\delta$, and $c_\gamma(\xi)\cap f_\beta(\xi)=c_\beta(\xi)$ holds for all $\beta\in\Lim(C_\gamma)$ and all $\xi_\gamma\leq\xi<\delta$. Moreover, if $\beta_0\in\Lim\cap\gamma$ and $\beta_1\in\Lim(C_\gamma)$ with $\beta_0<\beta_1$, then there is $\xi_\gamma\leq\zeta<\delta$ with $f_{\beta_0}(\xi)<f_{\beta_1}(\xi)$ and $c_{\beta_0}(\xi)=c_{\beta_1}(\xi)\cap f_{\beta_0}(\xi)$ for all $\zeta\leq\xi<\delta$ and hence we have $f_{\beta_0}(\xi)<f_\gamma(\xi)$ and $c_{\beta_0}(\xi)=c_\gamma(\xi)\cap f_{\beta_0}(\xi)$ for all $\zeta\leq\xi<\delta$. 
 
   \begin{case}
     $\gamma \in E$. 
   \end{case}
 
   Fix $\alpha<\kappa^+$ with $\gamma=\gamma_\alpha$. 
      Let $\seq{\beta_\xi}{\xi<\delta}$ be the monotone enumeration of a subset of $\Lim(C_\gamma)$ of order-type $\delta$ that is closed unbounded in $\gamma$.  
   Given $\xi_\gamma\leq\xi<\delta$,  we have $f_{\beta_\xi}(\xi)<f_\gamma(\xi)$ and we can therefore pick a closed unbounded subset $C^\gamma_\xi$ of $f_\gamma(\xi)$ with $\min(C^\gamma_\xi)=f_{\beta_\xi}(\xi)$ and $C^\gamma_\xi\neq h_\alpha(\xi)\cap[f_{\beta_\xi}(\xi),f_\gamma(\xi))$.  
   Now, define 
   \begin{equation*}
     c_\gamma(\xi) ~ = ~ \begin{cases} 
      \omega, & \textit{for all $\xi<\xi_\gamma$.} \\ 
      c_{\beta_\xi}(\xi) ~ %\cup ~ \{f_{\beta_\xi}(\xi)\} ~ 
      \cup ~ C^\gamma_\xi, & \textit{for all $\xi_\gamma\leq\xi<\delta$.} 
  \end{cases}
 \end{equation*}

%$c_\gamma(\xi)=\omega$ for all $\xi<\xi_\gamma$ and $$c_\gamma(\xi) ~ = ~ c_{\beta_\xi}(\xi) ~ \cup ~ \{f_{\beta_\xi}(\xi)\} ~ \cup ~ C^\gamma_\xi$$ for all $\xi_\gamma\leq\xi<\delta$. 
    %
   Then $c_\gamma(\xi)$ is a closed unbounded subset of $f_\gamma(\xi)$ for all $\xi<\delta$ and, if $\xi_\gamma\leq\xi<\delta$, then $c_\gamma(\xi)\neq h_\alpha(\xi)\cap f_\gamma(\xi)$.  
   Moreover, if $\xi_\gamma\leq\zeta<\xi<\delta$, then  $\beta_\zeta\in\Lim(C_{\beta_\xi})$, $\beta_\xi\notin E$, $\xi>\xi_{\beta_\xi}$, $f_{\beta_\zeta}(\xi)<f_{\beta_\xi}(\xi)<f_\gamma(\xi)$ and $$c_{\beta_\zeta}(\xi) ~ = ~ c_{\beta_\xi}(\xi)\cap f_{\beta_\zeta}(\xi) ~ = ~ f_\gamma(\xi)\cap f_{\beta_\zeta}(\xi).$$  
   In particular, this shows that for all $\beta\in\Lim\cap\gamma$, we have $f_\beta(\xi)<f_\gamma(\xi)$ and $c_\beta(\xi)=c_\gamma(\xi)\cap f_\beta(\xi)$ for coboundedly many $\xi<\delta$. 
   
 \medskip
  
  The above construction ensures that $[c_\gamma]_U$ is a closed unbounded subset of $[f_\gamma]_U$ for all $\gamma\in\Lim\cap\kappa^+$. Moreover, we have $[c_\beta]_U=[c_\gamma]_U\cap[f_\beta]_U$ for all $\beta,\gamma\in\Lim\cap\kappa^+$ with $\beta<\gamma$. 
  In particular, there is a closed unbounded subset $C$ of $\kappa^+$ with $C\cap[f_\gamma]_U=[c_\gamma]_U$ for all $\gamma\in\Lim\cap\kappa^+$.

  \begin{claim*}
   The set $C$ is fresh over $\Ult{\VV}{U}$. 
  \end{claim*}
  
  \begin{proof}[Proof of the Claim]
   First, if $\beta<\kappa^+$, then there is $\gamma\in\Lim\cap\kappa^+$ with $[f_\gamma]_U>\beta$  and  $$C\cap\beta ~ =  ~ (C\cap[f_\gamma]_U)\cap\beta ~ = ~ [c_\gamma]_U\cap\beta ~ \in ~ \Ult{\VV}{U}.$$ 
   
   Next, assume, towards a contradiction, that $C$ is an element of $\Ult{\VV}{U}$. 
   Then there is an $\alpha<\kappa^+$ with $C=[h_\alpha]_U$. 
    Since we have $$[h_\alpha]_U\cap[f_{\gamma_\alpha}]_U ~ = ~ C\cap[f_{\gamma_\alpha}]_U ~ = ~ [c_{\gamma_\alpha}]_U,$$ we know that the set $\Set{\xi<\delta}{h_\alpha(\xi)\cap f_{\gamma_\alpha}(\xi)=c_{\gamma_\alpha}(\xi)}$ is an element of $U$. 
    In particular, we can find $\xi_{\gamma_\alpha}\leq\xi<\delta$ with $h_\alpha(\xi)\cap f_{\gamma_\alpha}(\xi)=c_{\gamma_\alpha}(\xi)$, contradicting the definition of $c_{\gamma_\alpha}$. 
      \end{proof}
  
  This completes the proof of the theorem. 
\end{proof}

%%%%%%%%%%%%%%%%%%%%%%%%%%%%%%%%%%%%%

\section{Regular cardinals in $\Ult{\VV}{U}$}

 We now turn to the construction of fresh subsets of limit ordinals that are not cardinals in $\VV$. 
 We first observe that we can restrict ourselves to ordinals that are regular cardinals in the corresponding ultrapower.

\begin{proposition}\label{proposition:CofUltraFresh}
 Let $U$ be a normal ultrafilter on a measurable cardinal $\delta$ and let $\lambda$ be a limit cardinal. If there is an unbounded subset of $\cof{\lambda}^{\Ult{\VV}{U}}$ that is fresh over $\Ult{\VV}{U}$, then there is an unbounded subset of $\lambda$ that is fresh over $\Ult{\VV}{U}$. 
\end{proposition}

\begin{proof}
 Set $\lambda_0=\cof{\lambda}^{\Ult{\VV}{U}}$. Let $A$ be an unbounded subset of $\lambda_0$ that is fresh over $\Ult{\VV}{U}$ and let $\seq{\gamma_\eta}{\eta<\lambda_0}$ be a strictly increasing sequence that is cofinal in $\lambda$ and an element of $\Ult{\VV}{U}$. In this situation, the set $\Set{\gamma_\eta}{\eta\in A}$ is unbounded in $\lambda$ and fresh over $\Ult{\VV}{U}$. 
\end{proof}

In the proof of the following result, we modify techniques from the proof of Theorem \ref{theorem:SuccSingularFresh} to cover the non-cardinal case in Theorem \ref{theorem:CanonicalModelsCharFresh}.

\begin{theorem}\label{theorem:FreshLimitOrdinalsNotCardinals}
 Let $U$ be a normal ultrafilter on a measurable cardinal $\delta$, let $\kappa$ be a singular cardinal of cofinality $\delta$ with the property that $\mu^\delta<\kappa$ holds for all $\mu<\kappa$ and let $\kappa^+<\lambda<j_U(\kappa)$ be a limit ordinal of cofinality $\kappa^+$ that is a regular cardinal in $\Ult{\VV}{U}$. If there is a $\square_\kappa$-sequence, then there is an unbounded subset of $\lambda$ that is fresh over $\Ult{\VV}{U}$. 
\end{theorem}

\begin{proof}
 As in the proof of Theorem \ref{theorem:SuccSingularFresh}, we can apply Lemma \ref{lemma:UltrapowerEmbeddingFixedPoints} to find the monotone enumeration $\seq{\kappa_\xi}{\xi<\delta}$ of a closed unbounded subset of $\kappa$ of order-type $\delta$ with the property that $j_U(\kappa_\xi)=\kappa_\xi$ holds for all $\xi<\delta$. 
 Then normality implies that $[\xi\mapsto\kappa_\xi]_U=\kappa$ and we can repeat arguments from the first part of the proof of Theorem \ref{theorem:SuccSingularFresh} to see that $[\xi\mapsto\kappa_\xi^+]_U=\kappa^+$. 
 By our assumptions, there is a function $h$ with domain $\delta$, $[h]_U=\lambda$ and the property that $h(\xi)$ is a regular cardinal in the interval $(\kappa_\xi^+,\kappa)$ for all $\xi<\delta$. Fix a sequence $\seq{h_\gamma\in\prod_{\xi<\delta}h(\xi)}{\gamma<\kappa^+}$ such that the sequence $\seq{[h_\gamma]_U}{\gamma<\kappa^+}$ is strictly increasing and cofinal in $\lambda$. 

 Pick a $\square_\kappa$-sequence $\seq{C_\gamma}{\gamma\in\Lim\cap\kappa^+}$ with $\otp{C_\gamma}<\kappa$ for all $\gamma\in\Lim\cap\kappa^+$. 
  Given $\gamma\in\Lim\cap\kappa^+$, we let $\xi_\gamma$ denote the minimal element $\xi$ of $\delta$ with $\kappa_\xi^+>\otp{C_\gamma}$.  
 
 We now inductively construct a sequence $$\seq{f_\gamma\in\prod_{\xi<\delta}h(\xi)}{\gamma<\kappa^+}$$ by setting: 
  \begin{itemize}
   \item $f_0(\xi)=0$ for all $\xi<\delta$. 
   
   \item $f_{\gamma+1}(\xi)=\max(f_\gamma(\xi),h_\gamma(\xi))+1$ for all $\gamma<\kappa^+$ and $\xi<\delta$. 
   
   \item If $\gamma\in\Lim\cap\kappa^+$ with $\Lim(C_\gamma)$ bounded in $\gamma$ and $\xi<\delta$, then 
    %$$f_\gamma(\xi) ~ = ~ \sup\Set{f_\beta(\xi)}{\beta\in C_\gamma\setminus\max(\Lim(C_\gamma))}.$$
   $$f_\gamma(\xi) ~ = ~ \min\Set{\lambda\in\Lim}{\textit{$\lambda>f_\beta(\xi)$ for all $\beta\in C_\gamma\setminus\max(\Lim(C_\gamma))$}}.$$ 
   
   \item If $\gamma\in\Lim\cap\kappa^+$ with $\Lim(C_\gamma)$ unbounded in $\gamma$ and $\xi<\xi_\gamma$, then $f_\gamma(\xi)=\omega$. 
   
      \item If $\gamma\in\Lim\cap\kappa^+$ with $\Lim(C_\gamma)$ unbounded in $\gamma$ and $\xi_\gamma\leq\xi<\delta$, then $$f_\gamma(\xi) ~ = ~ \sup\Set{f_\beta(\xi)}{\beta\in\Lim(C_\gamma)}.$$ 
  \end{itemize}

  As in the proof of Theorem \ref{theorem:SuccSingularFresh}, we have the following claim. 

  \begin{claim*}
   \begin{enumerate}
    \item If $\beta<\gamma<\kappa^+$, then $f_\beta(\xi)<f_\gamma(\xi)$ for coboundedly many $\xi<\delta$. 
    
    \item If $\gamma\in\Lim\cap\kappa^+$, then $f_\gamma(\xi)\in\Lim$ for all $\xi<\delta$. 
    
    \item If $\gamma\in\Lim\cap\kappa^+$, $\beta\in\Lim(C_\gamma)$ and $\xi_\gamma\leq\xi<\delta$, then $f_\beta(\xi)<f_\gamma(\xi)$. \qed 
   \end{enumerate}
  \end{claim*}

 In particular, this shows that the sequence $\seq{[f_\gamma]_U}{\gamma<\kappa^+}$ is strictly increasing. 
 Since the above definition ensures that $[h_\gamma]_U<[f_{\gamma+1}]_U$ holds for all $\gamma<\kappa^+$, we also know that this sequence is cofinal in $\lambda$.

  Next, we inductively define a sequence $\seq{c_\gamma}{\gamma\in\Lim\cap\kappa^+}$ of functions with domain $\delta$ such that the following statements hold for all $\gamma\in\Lim\cap\kappa^+$: 
  \begin{enumerate}
   \item[(a)] $c_\gamma(\xi)$ is a closed unbounded subset of $f_\gamma(\xi)$ with $\otp{c_\gamma(\xi)}<\kappa_\xi^+$ for all $\xi<\delta$. 
   
%   \item[(b)] Given $\xi<\delta$ with $\kappa_\xi^+>\otp{C_\gamma}$, the set $c_\gamma(\xi)$ is unbounded in $f_\gamma(\xi)$ if and only if $\gamma\notin E$. 
   
   \item[(b)] If $\beta\in\Lim\cap\gamma$, then $f_\beta(\xi)<f_\gamma(\xi)$ and $c_\beta(\xi)=c_\gamma(\xi)\cap f_\beta(\xi)$ for coboundedly many $\xi<\delta$. 
   
   \item[(c)] If $\beta\in\Lim(C_\gamma)$ and $\xi_\gamma\leq\xi<\delta$, then $c_\beta(\xi)=c_\gamma(\xi)\cap f_\beta(\xi)$.  
  \end{enumerate}

  Our inductive construction distinguishes between the following cases:

   \setcounter{case}{0}
  
   \begin{case}
     $\gamma\in\Lim\cap\kappa^+$ with $\Lim\cap\gamma$ bounded in $\gamma$.
   \end{case}
   
  We set $\beta_0=0$ if $\Lim(C_\gamma)=\emptyset$ and $\beta_0=\max(\Lim(C_\gamma))$ otherwise. 
  Moreover, we set $\beta_1=0$ if $\gamma=\omega$ and $\beta_1=\max(\Lim\cap\gamma)$ otherwise. 
   This definition ensures that $\beta_0\leq\beta_1<\gamma$  and $f_{\beta_0}(\xi)<f_\gamma(\xi)$ for all $\xi<\delta$. 
   %      we have $\beta_0\leq\beta_1<\gamma$, $f_{\beta_0}(\xi)<f_\gamma(\xi)$ for all $\xi_\gamma\leq\xi<\delta$ and there is       %
  We can now find $\xi_\gamma\leq\zeta<\delta$ with the property that $f_{\beta_0}(\xi)\leq f_{\beta_1}(\xi)<f_\gamma(\xi)$ for all $\zeta\leq\xi<\delta$ and, if $\beta_0>0$, then $c_{\beta_0}(\xi)=c_{\beta_1}(\xi)\cap f_{\beta_0}(\xi)$ for all $\zeta\leq\xi<\delta$. 
   Since the definition of $f_\gamma$ implies that $\cof{f_\gamma(\xi)}=\omega$ holds for all $\xi<\delta$, we can pick a sequence of strictly increasing functions $\seq{\map{k_\xi}{\omega}{f_\gamma(\xi)}}{\xi<\delta}$ with the property that $k_\xi$ is cofinal in $f_\gamma(\xi)$  for all $\xi<\delta$, $k_\xi(0)=f_{\beta_0}(\xi)$ for all $\xi<\zeta$, and $k_\xi(0)=f_{\beta_1}(\xi)$ for all $\zeta\leq\xi<\delta$. 
   Let 
 \begin{equation*}
   c_\gamma(\xi) ~ = ~ \begin{cases}
  % \Set{k_\xi(n)}{n<\omega}, & \text{ for all } \xi<\xi_\gamma\\
  \Set{k_\xi(n)}{n<\omega}, & \textit{for all $\xi<\zeta$, if $\beta_0=0$.}\\
  c_{\beta_0}(\xi) ~ \cup ~ \Set{k_\xi(n)}{n<\omega}, & \textit{for all $\xi<\zeta$, if $\beta_0>0$.}\\
  \Set{k_\xi(n)}{n<\omega}, & \textit{for all $\zeta\leq\xi<\delta$, if $\beta_1=0$.}\\ 
  c_{\beta_1}(\xi) ~ \cup ~ \Set{k_\xi(n)}{n<\omega}, & \textit{for all $\zeta\leq\xi<\delta$, if $\beta_1>0$.}
  \end{cases}
 \end{equation*}

  %Set $c_\gamma(\xi)=\Set{k_\xi(n)}{n<\omega}$ for all $\xi<\xi_\gamma$. 
   %
  %Next, set $c_\gamma(\xi)=\Set{k_\xi(n)}{n<\omega}$ for all $\xi_\gamma\leq\xi<\zeta$ if $\beta_0=0$ and set $c_\gamma(\xi)=c_{\beta_0}(\xi)\cup\Set{k_\xi(n)}{n<\omega}$  for all $\xi_\gamma\leq\xi<\zeta$ if $\beta_0>0$.  
  %
  %Finally, set $c_\gamma(\xi)=\Set{k_\xi(n)}{n<\omega}$ for all $\zeta\leq\xi<\delta$ if $\beta_1=0$ and set $c_\gamma(\xi)=c_{\beta_1}(\xi)\cup\Set{k_\xi(n)}{n<\omega}$ for all $\zeta\leq\xi<\delta$ if $\beta_1>0$. 
  
 Then $c_\gamma(\xi)$ is a closed unbounded subset of $f_\gamma(\xi)$ of order-type less than $\kappa_\xi^+$ for all $\xi<\delta$ and, if $\beta_0>0$, then $c_\gamma(\xi)\cap f_{\beta_0}(\xi)=c_{\beta_0}(\xi)$ for all $\xi<\delta$. 
  In particular, we know that $c_\gamma(\xi)\cap f_\beta(\xi)=c_\beta(\xi)$ for all $\beta\in\Lim(C_\gamma)$ and all $\xi_\gamma\leq\xi<\delta$. 
  Finally, notice that $\beta_1>0$ implies that $f_{\beta_1}(\xi)<f_\gamma(\xi)$ and $c_\gamma(\xi)\cap f_{\beta_1}(\xi)=c_{\beta_1}(\xi)$ hold for all $\zeta\leq\xi<\delta$.  This shows that for all $\beta\in\Lim\cap\gamma$, we have $f_\beta(\xi)<f_\gamma(\xi)$ and $c_\beta(\xi)=c_\gamma(\xi)\cap f_\beta(\xi)$ for coboundedly many $\xi<\delta$.  

  \begin{case}
    $\gamma\in\Lim\cap\kappa^+$ with $\Lim\cap\gamma$ unbounded in $\gamma$ and $\Lim(C_\gamma)$ bounded in $\gamma$.
  \end{case}
  
 Since the limit points of $C_\gamma$ are bounded in $\gamma$, we have $\cof{\gamma}=\omega$ and we can pick a strictly increasing sequence $\seq{\beta_n}{n<\omega}$ cofinal in $\gamma$ such that $\beta_n\in\Lim\cap\gamma$ for all $0<n<\omega$, $\beta_0=0$ in case $\Lim(C_\gamma)=\emptyset$, and $\beta_0=\max(\Lim(C_\gamma))$ in case $\Lim(C_\gamma)\neq\emptyset$. 
  Fix $\xi_\gamma\leq\zeta<\delta$ such that $f_{\beta_n}(\xi)<f_{\beta_{n+1}}(\xi)<f_\gamma(\xi)$ and $c_{\beta_{n+2}}(\xi)\cap f_{\beta_{n+1}}(\xi)=c_{\beta_{n+1}}(\xi)$ for all $\zeta\leq\xi<\delta$ and all $n<\omega$, and, if $\beta_0>0$, $c_{\beta_1}(\xi)\cap f_{\beta_0}(\xi)=c_{\beta_0}(\xi)$ for all $\zeta\leq\xi<\delta$. 
  Then the definition of $f_\gamma$ ensures that $\cof{f_\gamma(\xi)}=\omega$ and $f_{\beta_0}(\xi)<f_\gamma(\xi)$ for all $\xi<\delta$. 
  Moreover, it also directly implies that $f_\gamma(\xi)=\sup\Set{f_{\beta_n}(\xi)}{n<\omega}$ holds  for all $\zeta\leq\xi<\delta$.  
  Fix a sequence of strictly increasing functions $\seq{\map{k_\xi}{\omega}{f_\gamma(\xi)}}{\xi<\zeta}$ such that $k_\xi(0)=f_{\beta_0}(\xi)$ and $k_\xi$ is cofinal  in $f_\gamma(\xi)$ for all $\xi<\zeta$. 
  Define 
  \begin{equation*}
   c_\gamma(\xi) ~ = ~ \begin{cases}
   %\Set{k_\xi(n)}{n<\omega}, & \text{ for all } \xi<\xi_\gamma\\
   \Set{k_\xi(n)}{n<\omega}, & \textit{for all $\xi<\zeta$, if $\beta_0=0$.}\\
   c_{\beta_0}(\xi) ~ \cup ~ \Set{k_\xi(n)}{n<\omega}, & \textit{for all $\xi<\zeta$, if $\beta_0>0$.}\\ 
   \bigcup\Set{c_{\beta_n}(\xi)}{0<n<\omega}, & \textit{for all $\zeta\leq\xi<\delta$,  if $\beta_0=0$.}\\
   \bigcup\Set{c_{\beta_n}(\xi)}{n<\omega}, & \textit{for all $\zeta\leq\xi<\delta$, if $\beta_0>0$.}
  \end{cases}
 \end{equation*}

 %Set $c_\gamma(\xi)=\Set{k_\xi(n)}{n<\omega}$ for all $\xi<\xi_\gamma$. 
 %
 %Next, set $c_\gamma(\xi)=\Set{k_\xi(n)}{n<\omega}$ for all $\xi_\gamma\leq\xi<\zeta$ if $\beta_0=0$ and set $c_\gamma(\xi)=c_{\beta_0}(\xi)\cup\Set{k_\xi(n)}{n<\omega}$ for all $\xi_\gamma\leq\xi<\zeta$ if $\beta_0>0$. 
  %
 % Finally, set $c_\gamma(\xi)=\bigcup\Set{c_{\beta_n}(\xi)}{0<n<\omega}$ for all $\zeta\leq\xi<\delta$ if $\beta_0=0$ and set $c_\gamma(\xi)=\bigcup\Set{c_{\beta_n}(\xi)}{n<\omega}$ for all $\zeta\leq\xi<\delta$  if $\beta_0>0$.  
  
  Given $\xi<\delta$, the set $c_\gamma(\xi)$ is closed and unbounded in $f_\gamma(\xi)$ and the regularity of $\kappa_\xi^+$ implies that $\otp{c_\gamma(\xi)}<\kappa_\xi^+$. 
  Next, $\beta_0>0$ implies that  $c_\gamma(\xi)\cap f_{\beta_0}(\xi)=c_{\beta_0}(\xi)$ for all $\xi<\delta$, and therefore $c_\gamma(\xi)\cap f_\beta(\xi)=c_\beta(\xi)$ for all $\beta\in\Lim(C_\gamma)$ and all $\xi_\gamma\leq\xi<\delta$. 
  Finally, we have $f_{\beta_n}(\xi)<f_\gamma(\xi)$ and $c_\gamma(\xi)\cap f_{\beta_n}(\xi)=c_{\beta_n}(\xi)$ for all $0<n<\omega$ and $\zeta\leq\xi<\delta$, and hence for all $\beta\in\Lim\cap\gamma$, we have $f_\beta(\xi)<f_\gamma(\xi)$ and $c_\gamma(\xi)\cap f_\beta(\xi)=c_\beta(\xi)$ for coboundedly many $\xi<\delta$. 
  
  \begin{case}
    $\gamma\in\Lim\cap\kappa^+$ with $\Lim(C_\gamma)$ unbounded in $\gamma$.
  \end{case}

   Let 
   \begin{equation*}
    c_\gamma(\xi) ~ = ~ \begin{cases} 
    \omega, & \textit{for all $\xi<\xi_\gamma$.}\\
    \bigcup\Set{c_\beta(\xi)}{\beta\in\Lim(C_\gamma)}, & \textit{for all $\xi_\gamma\leq\xi<\delta$.}
   \end{cases}
  \end{equation*}

  % Set $c_\gamma(\xi)=\omega$ for all $\xi<\xi_\gamma$ and $$c_\gamma(\xi) ~ = ~ \bigcup\Set{c_\beta(\xi)}{\beta\in\Lim(C_\gamma)}$$ for all $\xi_\gamma\leq\xi<\delta$. 
   %
 Given $\beta_0,\beta_1\in\Lim(C_\gamma)$ with $\beta_0<\beta_1$ and  $\xi_\gamma\leq\xi<\delta$, the above definition ensures that $f_{\beta_0}(\xi)<f_{\beta_1}(\xi)$ and $c_{\beta_0}(\xi)=c_{\beta_1}(\xi)\cap f_{\beta_0}(\xi)$. 
 Since we have $f_\gamma(\xi)=\sup\Set{f_\beta(\xi)}{\beta\in\Lim(C_\gamma)}$ and $\otp{C_\gamma}<\kappa_\xi^+$ for all $\xi_\gamma\leq\xi<\delta$, this shows that $c_\gamma(\xi)$ is a closed unbounded subset of $f_\gamma(\xi)$ of order-type less than $\kappa_\xi^+$ for all $\xi<\delta$, and, if $\beta\in\Lim(C_\gamma)$ and $\xi_\gamma\leq\xi<\delta$, then $c_\gamma(\xi)\cap f_\beta(\xi)=c_\beta(\xi)$ holds. 
    Finally, given $\beta_0\in\Lim\cap\gamma$ and $\beta_1\in\Lim(C_\gamma)$ with $\beta_0<\beta_1$, our induction hypothesis yields $\xi_\gamma\leq\zeta<\delta$ with $f_{\beta_0}(\xi)<f_{\beta_1}(\xi)$ and $c_{\beta_0}(\xi)=c_{\beta_1}(\xi)\cap f_{\beta_0}(\xi)$ for all $\zeta\leq\xi<\delta$, and this ensures that $f_{\beta_0}(\xi)<f_\gamma(\xi)$ and $c_{\beta_0}(\xi)=c_\gamma(\xi)\cap f_{\beta_0}(\xi)$ for all $\zeta\leq\xi<\delta$. 
 
 \medskip
 
 Given $\gamma\in\Lim\cap\kappa^+$, the properties listed above ensure that $[c_\gamma]_U$ is a closed unbounded subset of $[f_\gamma]_U$ of order-type less than $\kappa^+$. Moreover, if $\beta,\gamma\in\Lim\cap\kappa^+$ with $\beta<\gamma$, then $[c_\beta]_U=[c_\gamma]_U\cap[f_\beta]_U$. These observations show that there is a closed unbounded subset $C$ of $\lambda$ with $C\cap[f_\gamma]_U=[c_\gamma]_U$ for all $\gamma\in\Lim\cap\kappa^+$ and this property directly implies that $\otp{C}=\kappa^+<\lambda$. Since $\lambda$ is a regular cardinal in $\Ult{\VV}{U}$, this allows us to conclude that the set $C$ is not contained in $\Ult{\VV}{U}$ and hence it is fresh over $\Ult{\VV}{U}$. 
 \end{proof}

We end this section by using the above results to show that the validity of the equivalence stated in Theorem \ref{theorem:sakai} has high consistency strength.

\begin{proof}[Proof of Theorem \ref{theorem:StrengthSakai}]
  Let $U$ be a normal ultrafilter on a measurable cardinal $\delta$. 
  
   (i) Assume that  $\kappa>\delta^+$ is a regular cardinal such that  there  exists $\theta\in\{\kappa,\kappa^+\}$ with the property that there is a $\square(\theta)$-sequence. The regularity of $\theta$ then implies that $j_U[\theta]$ is cofinal in $j_U(\theta)$ and hence $\cof{j_U(\theta)}=\theta$. Using Theorem \ref{theorem:IndSquareFresh}, we now find an unbounded subset of $j_U(\theta)$ that is fresh over $\Ult{\VV}{U}$.

  (ii) Now, assume that $\kappa$ is a singular strong limit cardinal of cofinality $\delta$ such that $\kappa^\delta=\kappa^+$ and there exists a $\square_\kappa$-sequence. 
  Set $\theta=\kappa^+$ and $\lambda=(\theta^+)^{\Ult{\VV}{U}}$.   
  Since $\betrag{j_U(\kappa)}\leq\kappa^\delta=\theta$ and elementarity implies that $\lambda<j_U(\kappa)$, we have $\lambda<\theta^+$. 
  
  \begin{claim*}
   $\cof{\lambda}=\theta$. 
  \end{claim*}
  
  \begin{proof}[Proof of the Claim]
    As in the proof of Theorem \ref{theorem:FreshLimitOrdinalsNotCardinals}, we can find a monotone enumeration $\seq{\kappa_\xi}{\xi<\delta}$ of a closed unbounded subset of $\kappa$ of order-type $\delta$ with the property that $[\xi\mapsto\kappa_\xi]_U=\kappa$ and $[\xi\mapsto\kappa_\xi^+]_U=\theta$. Then $[\xi\mapsto\kappa_\xi^{++}]_U=\lambda$. 
   Assume, towards a contradiction, that $\cof{\lambda}\neq\theta$. Since $\kappa$ is singular and $\lambda<\theta^+$, this show that there is $\zeta<\delta$ with $\cof{\lambda}<\kappa_\zeta$. 
   Pick a sequence $\seq{f_\alpha}{\alpha<\cof{\lambda}}$ of functions with domain $\delta$ such that $f_\alpha(\xi)<\kappa_\xi^{++}$ holds for all $\alpha<\cof{\lambda}$ and all $\xi<\delta$, and the induced sequence $\seq{[f_\alpha]_U}{\alpha<\cof{\lambda}}$  is strictly increasing and cofinal in $\lambda$. 
    By our assumption, there is a function $f$ with domain $\delta$ and the property that $f_\alpha(\xi)<f(\xi)<\kappa_\xi^{++}$ holds for all $\alpha<\cof{\lambda}$ and all $\zeta\leq\xi<\delta$. 
    But then we have $[f_\alpha]_U<[f]_U<\lambda$ for all $\alpha<\cof{\lambda}$, a contradiction. 
  \end{proof}
  
  Since $\kappa$ is a strong limit cardinal, the above claim now allows us to apply Theorem \ref{theorem:FreshLimitOrdinalsNotCardinals} to find an unbounded subset of $\lambda$ that is fresh over $\Ult{\VV}{U}$. 
  % But this contradicts our assumptions, because $\cof{\lambda}=\theta\neq\delta^+$.   
\end{proof}

%%%%%%%%%%%%%%%%%%%%%%%%%%%%%%%%%%%%%

\section{Ultrapowers of canonical inner models}

With the help of the results of the previous sections, we are now ready to prove the main result of this paper.

\begin{proof}[Proof of Theorem \ref{theorem:CanonicalModelsCharFresh}]
 Fix a normal ultrafilter $U$ on a measurable cardinal $\delta$ that satisfies the three assumptions listed in the statement of the theorem.  
 By Proposition \ref{proposition:MinNonFresh}, if $\lambda$ is a limit ordinal with the property that the cardinal $\cof{\lambda}$ is either smaller than $\delta^+$ or weakly compact, then no unbounded subset of $\lambda$ is fresh over $\Ult{\VV}{U}$. In the proof of the converse implication, we first consider two special cases.

 \begin{claim*}
   If $\kappa$ is a cardinal with the property that the cardinal $\cof{\kappa}$ is greater than $\delta$ and not weakly compact, then there is an unbounded subset of $\kappa$ that is fresh over $\Ult{\VV}{U}$. 
 \end{claim*}
 
 \begin{proof}[Proof of the Claim]
  We start by noting that, if $\cof{\kappa}=\delta^+$, then the fact that our assumptions imply that $2^\delta=\delta^+$ holds allows us to use Proposition \ref{proposition:MinFresh}  find a subset of $\kappa$ with the desired properties. 
  Therefore, in the following, we may assume that $\delta^+<\cof{\kappa}\leq\kappa$. Let $\nu\leq\kappa$ be minimal with $\nu^\delta\geq\kappa$. 
   By the minimality of $\nu$, we then have $\mu^\delta<\nu$ for all $\mu<\nu$. 
   In particular, the fact that $2^\delta=\delta^+<\kappa$ implies that $\nu>2^\delta>\delta$ and therefore we know that $\nu^+=\nu^\nu\geq \nu^\delta\geq\kappa\geq\nu$. 
   These computations show that either $\cof{\nu}>\delta$ and $\kappa=\nu$, or $\cof{\nu}\leq\delta$ and $\kappa=\nu^+$. 
   %
 %  Since $\cof{\kappa}>\delta^+$, the $\SCH$ allows us to apply {\cite[Theorem 5.22]{MR1940513}} to conclude that either , or . 

  First, assume that either $\cof{\nu}>\delta$ and $\kappa=\nu$, or $\cof{\nu}<\delta$ and $\kappa=\nu^+$. 
  Then Lemma \ref{lemma:UltrapowerEmbeddingFixedPoints} shows that $j_U(\kappa)=\kappa$ holds in both cases. 
  Moreover, since $\cof{\kappa}$ is a regular cardinal greater than $\delta^+$ 
  %, we have $j_U(\cof{\kappa})=\sup(j_U[\cof{\kappa}])$ \todo{See also comment at the statement of Theorem 3.2} 
  and $$j_U(\cof{\kappa}) ~ = ~ \cof{j_U(\kappa)}^{\Ult{\VV}{U}} ~ = ~ \cof{\kappa}^{\Ult{\VV}{U}},$$ the fact that $\cof{\kappa}$ is not weakly compact allows us to use Theorem \ref{theorem:IndSquareFresh} to find an unbounded subset of $\cof{\kappa}^{\Ult{\VV}{U}}$ that is fresh over $\Ult{\VV}{U}$.
  In this situation,  we can then apply Proposition \ref{proposition:CofUltraFresh} to obtain an unbounded subset of $\kappa$ that is fresh over $\Ult{\VV}{U}$.

  Finally, assume that $\cof{\nu}=\delta$ and $\kappa=\nu^+$. In this situation, we know that $\nu$ is a singular cardinal of cofinality $\delta$ with $2^\nu=\nu^+$ and the property that $\mu^\delta<\nu$ holds for all $\mu<\nu$. Since the assumptions of the theorem guarantee the existence of a $\square_\nu$-sequence, we can apply Theorem \ref{theorem:SuccSingularFresh}  to find an unbounded subset of $\kappa$ that is fresh over $\Ult{\VV}{U}$. 
 \end{proof}

 \begin{claim*}
  Let $\lambda$ be a limit ordinal with the property that the cardinal $\cof{\lambda}$ is greater than $\delta$ and not weakly compact. If $\lambda$ is a regular cardinal in $\Ult{\VV}{U}$, then there is an unbounded subset of $\lambda$ that is fresh over $\Ult{\VV}{U}$. 
 \end{claim*}
 
 \begin{proof}[Proof of the Claim]
  First, if $\lambda$ is a cardinal, then we can use the above claim to directly derive the desired conclusion. Hence, we may assume that $\lambda$ is not a cardinal.

  \begin{subclaim*}
   %Assume that the $\SCH$ holds. Let $U$ be a normal ultrafilter on a measurable cardinal $\delta$ such that  $2^\delta=\delta^+$ holds and for every singular cardinal $\kappa>\delta$, there exists a $\square_\kappa$-sequence. If $\lambda$ is a limit ordinal with $\cof{\lambda}>\delta$ that is not a cardinal. If $\lambda$ is a regular cardinal in $\Ult{\VV}{U}$, then 
   There is a cardinal $\kappa$ of cofinality $\delta$ such that $$\lambda ~ \in ~ (\kappa^+,j_U(\kappa)] ~ \cup  ~ \{j_U(\kappa^+)\}$$ and $\kappa>\delta$ implies that $\mu^\delta<\kappa$ for all $\mu<\kappa$. 
 
  \end{subclaim*}
 
 \begin{proof}[Proof of the Subclaim]
  Let $\theta = |\lambda|$. Then our assumptions imply that $$\delta ~ < ~ \cof{\lambda} ~ \leq ~ \theta ~ < ~ \lambda ~ < ~ \theta^+.$$  
  Moreover, we have $\cof{\theta}\neq\delta$, because otherwise $\theta$ would be a singular strong limit cardinal of cofinality $\delta$  and  our assumptions would allow us to  repeat the argument from the first part of the proof of Theorem \ref{theorem:SuccSingularFresh} to show that $\theta^+=(\theta^+)^{\Ult{\VV}{U}}$, contradict our assumption that $\lambda$ is a cardinal in $\Ult{\VV}{U}$. 
  In addition, we know that there is some $\nu<\theta$ satisfying $\nu^\delta\geq\theta$, because otherwise Lemma \ref{lemma:UltrapowerEmbeddingFixedPoints} would imply that $j_U(\theta)=\theta<\lambda<\theta^+=j_U(\theta^+)=(j_U(\theta)^+)^{\Ult{\VV}{U}}$, which again contradicts the assumption that $\lambda$ is a cardinal in $\Ult{\VV}{U}$. 
  Let $\kappa<\theta$ be the minimal cardinal with the property that $\kappa^\delta\geq\theta$ holds. 
   Then the minimality of $\kappa$ implies that $\mu^\delta<\kappa$ holds for all $\mu<\kappa$. 
   
   First, assume that $\kappa=2$. Then $\delta<\theta\leq 2^\delta=\delta^+$ and therefore $\theta = \delta^+$. 
   Since Lemma \ref{lemma:UltrapowerEmbeddingFixedPoints} implies that $j_U(\delta^{++})=\delta^{++}$, we know that $j_U(\delta^{++})=\delta^{++}=\theta^+>\lambda$ and, as above, we can conclude that   $\lambda$ is not contained in the interval $(j_U(\delta^+),\delta^{++})$.  
    Moreover, since our assumptions on $\lambda$ directly imply that $\lambda$ is not contained in the interval $(j_U(\delta),j_U(\delta^+))$, we can conclude that $\lambda$ is an element of the set $(\delta^+,j_U(\delta)]\cup\{j_U(\delta^+)\}$ in this case.

  Next, assume that $\kappa>2^\delta$. 
  %   By {\cite[Theorem 5.22]{MR1940513}}, the $\SCH$ now     %
  Then our cardinal arithmetic assumptions and the minimality of $\kappa$ imply that $\cof{\kappa}\leq\delta$ and $\theta=\kappa^+$. 
 But then we already know that $\cof{\kappa}=\delta$, because otherwise we could apply Lemma \ref{lemma:UltrapowerEmbeddingFixedPoints}  to conclude that $j_U(\theta) = \theta < \lambda < \theta^+  =  j_U(\theta^+)$.
 Since our assumptions imply that $(\kappa^+)^\delta=\kappa^+$, Lemma \ref{lemma:UltrapowerEmbeddingFixedPoints} implies that $j_U(\kappa^{++})=\kappa^{++}=\theta^+$ and this shows that $\lambda$ is not contained in the interval $(j_U(\kappa^+),\kappa^{++})$.  
%   As $2^\delta < \kappa < \kappa^{++}$ and $\cof{\kappa^{++}} = \kappa^{++} > \delta$, we have $(\kappa^{++})^\delta=\kappa^{++}$ by the $\SCH$ and {\cite[Theorem 5.22]{MR1940513}}. Therefore, 
 %
   Since $\lambda$ is also not contained in the interval $(j_U(\kappa),j_U(\kappa^+))$, we can conclude that $\lambda$ is contained in the set $(\kappa^+,j_U(\kappa)]\cup\{j_U(\kappa^+)\}$. 
 \end{proof}

  First, assume that $\kappa=\delta$ holds. By our assumptions, Lemma \ref{lemma:UltrapowerEmbeddingFixedPoints} shows that $\delta^{++}=j_U(\delta^{++})>j_U(\delta^+)$. 
  Since we know that $\delta^+<\lambda\leq j_U(\delta^+)$ and $\cof{\lambda} > \delta$, this implies that $\cof{\lambda}=\delta^+$,  and hence we can use  Proposition \ref{proposition:MinFresh} to find an unbounded subset of $\lambda$ that is fresh over $\Ult{\VV}{U}$.    %

  Next, assume that $\kappa>\delta$ and $\lambda=j_U(\kappa)$. 
  Then $$\delta ~ < ~ \cof{\lambda} ~ \leq ~ \cof{\lambda}^{\Ult{\VV}{U}} ~ = ~ j_U(\cof{\kappa}) ~ = ~ j_U(\delta) ~ < ~ \delta^{++}$$ and we can conclude that $\cof{\lambda}=\delta^+$.
  Another application of Proposition \ref{proposition:MinFresh} now yields the desired subset of $\lambda$.

  Now, assume that $\kappa>\delta$ and $\lambda=j_U(\kappa^+)$.   
  Then our assumptions ensure the existence  a $\square(\kappa^+)$-sequence and therefore we can apply  Theorem \ref{theorem:IndSquareFresh} to find an unbounded subset of $\lambda$ that is fresh over $\Ult{\VV}{U}$.

  Finally, we assume that $\kappa>\delta$ and $\kappa^+<\lambda<j_U(\kappa)$. Then we know that $\mu^\delta<\kappa$ holds for all $\mu<\kappa$.

  \begin{subclaim*}
   $\cof{\lambda}=\kappa^+$. 
  \end{subclaim*}
  
  \begin{proof}[Proof of the Subclaim]
    Assume, towards a contradiction, that $\cof{\lambda}\neq\kappa^+$. 
    Since $\kappa$ is singular and $\cof{\lambda}<\lambda<j_U(\kappa)<\kappa^{++}$,  this implies that $\cof{\lambda}<\kappa$. 
     In this situation, we can repeat an argument from the first part of the proof of Theorem \ref{theorem:SuccSingularFresh} to find a  monotone enumeration $\seq{\kappa_\xi}{\xi<\delta}$ of a closed unbounded subset of $\kappa$ of order-type $\delta$ such that $\kappa_0>\cof{\lambda}$, $[\xi\mapsto\kappa_\xi]_U=\kappa$ and $[\xi\mapsto\kappa_\xi^+]_U=\kappa^+$. 
  Fix a function $f$ with domain $\delta$ such that $[f]_U=\lambda$ holds and $f(\xi)$ is a regular cardinal in the interval $(\kappa_\xi^+,\kappa)$ for all $\xi<\delta$.     
   Pick a sequence $\seq{f_\alpha}{\alpha<\cof{\lambda}}$ of functions with domain $\delta$ such that $f_\alpha(\xi)<f(\xi)$ holds for all $\alpha<\cof{\lambda}$ and all $\xi<\delta$, and the induced sequence $\seq{[f_\alpha]_U}{\alpha<\cof{\lambda}}$  is strictly increasing and cofinal in $\lambda$. 
   Given $\xi<\delta$, the fact that $f(\xi)$ is a regular cardinal greater than $\cof{\lambda}$ then yields an ordinal $\gamma_\xi<f(\xi)$ with $f_\alpha(\xi)<\gamma_\xi$ for all $\alpha<\cof{\lambda }$.  
 But then $[f_\alpha]_U<[\xi\mapsto\gamma_\xi]_U<\lambda$ for all $\alpha<\cof{\lambda}$, a contradiction. 
 \end{proof}

 By the above computations, we now know that $\kappa$ is a singular cardinal of cofinality $\delta$ with the property that $\mu^\delta<\kappa$ holds for all $\mu<\kappa$, and $\lambda$ is a limit ordinal of cofinality $\kappa^+$ with $\kappa^+<\lambda<j_U(\kappa)$ that is a regular cardinal in $\Ult{\VV}{U}$. 
 Since our assumptions guarantee  the existence of a $\square_\kappa$-sequence, we can use Theorem \ref{theorem:FreshLimitOrdinalsNotCardinals} to show that there also exists an unbounded subset of $\lambda$ that is fresh over $\Ult{\VV}{U}$ in this case. 
 \end{proof}

 To conclude the proof of the theorem, fix a limit ordinal $\lambda$ with the property that the cardinal $\cof{\lambda}$ is greater than $\delta$ and not weakly compact. 
 Set $\lambda_0=\cof{\lambda}^{\Ult{\VV}{U}}$. By {\cite[Lemma 3.7.(ii)]{MR1940513}}, we then have $\cof{\lambda_0}=\cof{\lambda}$. Hence, we can use the previous claim to find an unbounded subset of $\lambda_0$ that is fresh over $\Ult{\VV}{U}$. Using Proposition \ref{proposition:CofUltraFresh}, we can conclude that there is an unbounded subset of $\lambda$ that is fresh over $\Ult{\VV}{U}$. 
\end{proof}

We end this section by using famous results of Schimmerling and Zeman to show that, in canonical inner models, the assumptions of Theorem \ref{theorem:CanonicalModelsCharFresh} are satisfied for all measurable cardinals.

\begin{proof}[Proof of Theorem \ref{theorem:InnerModels}]
  We argue that Jensen-style extender models without subcompact cardinals satisfy the statements (a), (b) and (c) listed in Theorem \ref{theorem:CanonicalModelsCharFresh}. 
   First, notice that the $\GCH$ holds in all of these models and hence statement  (a) is satisfied.  
  Next, recall that {\cite[Theorem 15]{MR2081183}}\footnote{Schimmerling's and Zeman's notion of \emph{Jensen core model} in \cite{MR2081183} agrees with our notion of Jensen-style extender model.} shows that, in Jensen-style extender models, a $\square_\nu$-sequence exists if and only if $\nu$ is not a subcompact cardinal. 
  In particular, we know that,  in Jensen-style extender models without subcompact cardinals, $\square(\nu^+)$-sequences exist for all infinite cardinals $\nu$. 
  Since {\cite[Theorem 0.1]{MR2563821}} yields the existence of $\square(\kappa)$-sequences for inaccessible cardinals $\kappa$ in the relevant models, we can conclude that statement  (b) holds in these models. 
 Finally, the validity of statement (c) in Jensen-style extender models without subcompact cardinals again follows from {\cite[Theorem 15]{MR2081183}}. 
\end{proof}

%%%%%%%%%%%%%%%%%%%%%%%%%%%%%%%%%%%%%

\section{Consistency strength}

We end this paper by establishing the equiconsistency stated in Theorem \ref{theorem:ConsStrength}.  We start by showing that the existence of a weakly compact cardinal above a measurable cardinal is a lower bound for the consistency of the corresponding statement. 
%determining the consistency strength of the failure of the equivalence stated in Theorem \ref{theorem:CanonicalModelsCharFresh}. The proof of the following result is based on the fact that for every assumption about the measurable cardinal made in Theorem \ref{theorem:CanonicalModelsCharFresh}, the consistency strength of a failure of the given statement is strictly larger than the existence of a measurable cardinal. 

\begin{theorem}
 Assume that there is no inner model with a weakly compact cardinal above a measurable cardinal. 
  If $U$ is a normal ultrafilter on a measurable cardinal $\delta$, then there is an unbounded subset of $\delta^{++}$ that is fresh over $\Ult{\VV}{U}$. 
  %Then the statements (i) and (ii) listed in Theorem \ref{theorem:CanonicalModelsCharFresh} are equivalent for every normal ultrafilter $U$ on a measurable cardinal $\delta$ and every limit ordinal $\lambda$.  
\end{theorem}

\begin{proof}
% Let $U$ be a normal ultrafilter on a measurable cardinal $\delta$. 
 By our assumptions, we can use the results of  \cite{zbMATH04197980} to show that $2^\delta=\delta^+$ holds. 
  %and the $\SCH$ holds. 
  %
 %  Moreover, since our assumptions imply that $0^\dagger$ does not exist, we can use the results of \cite{zbMATH03825788} and \cite{MR2081183} to conclude that $\square_\kappa$ holds for every singular cardinal $\kappa$. 
  %
  %Finally, let $\kappa>\delta^+$ be a regular cardinal. 
  Set $\kappa=\delta^{++}$. 
  Then our assumptions imply that $\kappa$ is not weakly compact in $\LL[U]$. 
  In this situation, we can construct a tail of a $\square(\kappa)$-sequence $\seq{C_\nu}{\xi < \nu < \kappa, ~ \nu \in \Lim}$ in $\LL[U]$ above some ordinal $\xi > \delta^+$ with $\xi < \kappa$, using the argument in {\cite[Section 6]{JensenFine}} for $\LL$. 
   A consequence of this proof, published by Todor\v{c}evi\'{c} in {\cite[1.10]{TodorcevicPartCountableOrdinals}}, but probably first noticed by Jensen (see {\cite[Theorem 2.5]{SchimmerlingCoherentSeq}} for a modern account), is that the sequence $\seq{C_\nu}{\xi < \nu < \kappa, ~ \nu \in \Lim}$ remains a tail of a $\square(\kappa)$-sequence in $\VV$. We can now easily extend this sequence to a $\square(\kappa)$-sequence $\seq{C_\nu}{\nu \in \Lim \cap \kappa}$ in $\VV$. 
   Since $2^\delta=\delta^+$ holds,  Lemma \ref{lemma:UltrapowerEmbeddingFixedPoints} shows that $j_U(\kappa)=\kappa$ and hence we can use Theorem \ref{theorem:IndSquareFresh} to find an unbounded subset of $\kappa$ that is fresh over $\Ult{\VV}{U}$. 
\end{proof}

We now use forcing to show that the above large cardinal assumption is also an upper bound for the consistency strength of the non-existence of fresh subsets at the double successor of a measurable cardinal. 
%how a failure of the equivalence considered in Theorem \ref{theorem:CanonicalModelsCharFresh} can be established. 
 %
 The following lemma is a reformulation and slight strengthening of {\cite[Lemma 3.5]{sakainote}}.
 The notion of \emph{$\lambda$-strategically closed partial orders} and the corresponding game $G_\lambda(\PPP)$ are introduced in {\cite[Definition 5.15]{MR2768691}}.

\begin{lemma}\label{lemma:ClosedPreservesFreshness}
 Let $U$ be a normal ultrafilter on a measurable cardinal $\delta$, let $\lambda$ be a limit ordinal with $\cof{\lambda}>\delta$ and let $A$ be an unbounded subset of $\lambda$ that is fresh over $\Ult{\VV}{U}$. If $\PPP$ is a $(\delta+1)$-strategically closed partial order, then $$\mathbbm{1}_\PPP\Vdash\anf{\check{A}\notin\Ult{\VV}{\check{U}}}.$$
\end{lemma}

\begin{proof}
 Assume, towards a contradiction, that there is a condition $p$ in $\PPP$ and a $\PPP$-name $\dot{f}$ for a function with domain $\delta$ with the property that, whenever $G$ is $\PPP$-generic over $\VV$ with $p\in G$, then $[\dot{f}^G]_U=A$ holds in $\VV[G]$. 
 As $\PPP$ is $(\delta+1)$-strategically closed, there is a condition $p_1$ in $\PPP$ below $p$ and a subset $X$ of $\delta$ with the property that, whenever $G$ is $\PPP$-generic over $\VV$ with $p_1\in G$, then $X=\Set{\xi<\delta}{\dot{f}^G(\xi)\in\VV}$.

 \begin{claim*}
  If $\xi\in\delta\setminus X$ and $q\leq_\PPP p_1$, then there is $\gamma<\lambda$ and conditions $r_0$ and $r_1$ in $\PPP$ below $q$ such that $r_0\Vdash_\PPP\anf{\check{\gamma}\in\dot{f}(\check{\xi})}$ and $r_1\Vdash_\PPP\anf{\check{\gamma}\notin\dot{f}(\check{\xi})}$. 
 \end{claim*}
 
 \begin{proof}[Proof of the Claim]
  If such a pair of conditions does not exist, then it is easy to check that the condition $q$ forces $\dot{f}(\check{\xi})$ to be equal to the set $$\Set{\gamma<\lambda}{\exists r\leq_\PPP q ~ r\Vdash_\PPP\anf{\check{\gamma}\in\dot{f}(\check{\xi})}},$$ contradicting our assumption that $\xi$ is not an element of $X$. 
 \end{proof}

 \begin{claim*}
  $X\in U$. 
 \end{claim*}
 
 \begin{proof}[Proof of the Claim]
  Assume, towards a contradiction, that $X$ is not an element of $U$. 
 Fix a winning strategy $\sigma$ for Player {\sf Even} in the game $G_{\delta+1}(\PPP)$, some sufficiently large regular cardinal $\theta$ and an elementary submodel $M$ of $\HH{\theta}$ of cardinality $\delta$ satisfying $(\delta+1)\cup\{\lambda,\sigma,\dot{f},p_1,A,U,X,\PPP\}\subseteq M$ and ${}^{{<}\delta}M\subseteq M$. 
 We define $\eta=\sup(\lambda\cap M)<\lambda$ and fix a function $h$ with domain $\delta$ such that $[h]_U=A\cap\eta$. 
 
 Note that, given a partial run of $G_{\delta+1}(\PPP)$ of even length less than $\delta$ that consists of conditions in $M$ and was played according to $\sigma$ by Player {\sf Even}, the given sequence is an element of $M$ and Player {\sf Even} responds to it with a move in $M$. 
  %
% , the function $\bar\sigma = \sigma \cap M$ is a winning strategy for Player {\sf Even} in $M$. 
 % When playing the game $G_{\delta+1}(\PPP)$ according to $\sigma$ for moves in $M$,  in fact plays according to $\bar\sigma$ and hence 
 %
 Therefore, if $\tau$ is a strategy for Player {\sf Odd} in $G_{\delta+1}(\PPP)$ that answers to sequences of conditions in $M$ by playing a condition in $M$ and $\seq{p_\xi}{\xi\leq\delta}$ is a run of $G_{\delta+1}(\PPP)$ played according to $\sigma$ and $\tau$, then $p_\xi\in M$ for all $\xi<\delta$. 
 Moreover, the previous claim allows us to use elementarity to show for every $\xi\in\delta\setminus X$ and every condition $q\in M\cap\PPP$ with $q\leq_\PPP p_1$, there is $\gamma\in M\cap\lambda$ and a condition $r\in M\cap\PPP$ with $r\leq_\PPP p$ and 
  \begin{equation}\label{equation:equivalencesForcingSplit}
   \gamma\in h(\xi) ~ \Longleftrightarrow ~ r\Vdash_\PPP\anf{\check{\gamma}\notin\dot{f}(\check{\xi})} ~ \Longleftrightarrow ~ \neg(r\Vdash_\PPP\anf{\check{\gamma}\in\dot{f}(\check{\xi})}). %~ \Longleftrightarrow ~ r\not\Vdash_\PPP\anf{\check{\gamma}\in\dot{f}(\check{\xi})}.
  \end{equation}
 
  Now, pick  a strategy $\tau$ for Player {\sf Odd} in $G_{\delta+1}(\PPP)$ with the following properties:
   \begin{itemize}
    \item $\tau$ plays the condition $p_1$ in move $1$. 
    
    \item Given $\xi\in X$, if Player {\sf Even} played a condition $q\in M\cap\PPP$ in move $(2+2\cdot\xi)$, then $\tau$ responds by also playing the condition $q$ in the next move. 
    
    \item Given $\xi\in\delta\setminus X$, if Player {\sf Even} played a condition $q\in M\cap\PPP$ in move $(2+2\cdot\xi)$, then $\tau$ responds by playing a condition $r\in M\cap\PPP$ with $r\leq_\PPP q$ such that the equivalences of (\ref{equation:equivalencesForcingSplit}) hold true for some $\gamma\in M\cap\lambda$. 
  \end{itemize}
  
  Let $\seq{p_\xi}{\xi\leq\delta}$ be the run of $G_{\delta+1}(\PPP)$ played according to $\sigma$ and $\tau$. By the above remarks, we then have $p_\xi\in M$ for all $\xi<\delta$. In particular, for every $\xi\in\delta\setminus X$, there exists $\gamma_\xi<\lambda$ with $$\gamma_\xi\in h(\xi) ~ \Longleftrightarrow ~ p_\delta\Vdash_\PPP\anf{\check{\gamma}_\xi\notin\dot{f}(\check{\xi})} ~ \Longleftrightarrow ~ \neg(p_\delta\Vdash_\PPP\anf{\check{\gamma}_\xi\in\dot{f}(\check{\xi})}).$$ %~ \Longleftrightarrow ~ p_\delta\not\Vdash_\PPP\anf{\check{\gamma}\in\dot{f}(\check{\xi})}.$$ 
 
  Let $G$ be $\PPP$-generic over $\VV$ with $p_\delta\in G$. Then the closure properties of $\PPP$ imply that $[h]_U=A\cap\eta$ holds in $\VV[G]$. Since $A=[\dot{f}^G]_U$ holds in $\VV[G]$, we know that the set  $$Y ~ = ~ \Set{\xi<\delta}{\textit{$h(\xi)$ is an initial segment of $\dot{f}^G(\xi)$}}$$ is an element of $U$. But then there is some $\xi\in Y\setminus X = Y \cap (\delta \setminus X)$ and our construction ensures that the ordinal $\gamma_\xi$ is contained in the symmetric difference of $h(\xi)$ and $\dot{f}^G(\xi)$, a contradiction. 
 \end{proof}
 
 Now, let $G$ be $\PPP$-generic over $\VV$ with $p_1\in G$. By the previous claim and the closure properties of $\PPP$, we can find a function $f$ with domain $\delta$ in $\VV$ such that $[f]_U=[\dot{f}^G]_U=A$ holds in $\VV[G]$. Since forcing with $\PPP$ adds no new functions from $\delta$ to the ordinals, we can conclude that $[f]_U=A$ also holds in $\VV$, a contradiction as $A$ was chosen to be fresh over $\Ult{\VV}{U}$. 
\end{proof}

The previous lemma now allows us to prove the following results that can be used to complete the proof of Theorem \ref{theorem:ConsStrength} by considering the case $\mu=\delta^+$. 
%show that the anti-large cardinal hypothesis of Theorem \ref{theorem:InnerModelConStrength} is optimal. 

\begin{theorem}\label{theorem:CollapseWCnefresh}
 Let $U$ be a normal ultrafilter on a measurable cardinal $\delta$, let $\mu>\delta$ be a regular cardinal, let $\WW$ be an inner model containing $U$ and let $\kappa>\mu$ be weakly compact in $\WW$. If $\VV$ is a $\Col{\mu}{{<}\kappa}^\WW$-generic extension of $\WW$, then no unbounded subset of $\kappa$ is fresh over $\Ult{\VV}{U}$. 
\end{theorem}

\begin{proof}
  Assume, towards a contradiction, that there is an unbounded subset $A$ of $\kappa$ that is fresh over $\Ult{\VV}{U}$. 
 Note that, in $\VV$, our assumptions imply that $\mu^\delta=\mu$ and hence Lemma \ref{lemma:UltrapowerEmbeddingFixedPoints} implies that $j_U(\kappa)=\kappa$. In particular, for every $\gamma<\kappa$, there is a function $f\in\HH{\kappa}$ with domain $\delta$ and $[f]_U=\gamma$. 
 By our assumptions, there exists $G$ $\Col{\mu}{{<}\kappa}^\WW$-generic over $\WW$ with $\VV=\WW[G]$ and hence we know that ${}^{{<}\mu}\WW\subseteq\WW$. 
  Moreover, since $\Col{\mu}{{<}\kappa}$ satisfies the $\kappa$-chain condition in $\WW$, there exist $\Col{\mu}{{<}\kappa}$-nice names $\dot{A}$ and $\dot{F}$ in $\WW$ such that $\dot{A}^G=A$ and $\dot{F}^G$ is a function with domain $\kappa$ and the property that for all $\gamma<\kappa$, the set $\map{\dot{F}^G(\gamma)}{\delta}{\HH{\kappa}\cap\POT{\kappa}}$ is a function with $[\dot{F}^G(\gamma)]_U=A\cap\gamma$. 
  
  Work in $\WW$ and pick an elementary submodel $M$ of $\HH{\kappa^+}$ of cardinality $\kappa$ such that ${}^{{<}\kappa}M\subseteq M$ and $(\kappa+1)\cup\{\dot{A},\dot{F},U\}\subseteq M$. In this situation, the weak compactness of $\kappa$ yields a transitive set $N$ 
  %of $\ZFC^-$ (i.e., $\ZFC$ without the power set axiom) 
  with ${}^{{<}\kappa}N\subseteq N$ and an elementary embedding $\map{j}{M}{N}$ with critical point $\kappa$ (see {\cite[Theorem 1.3]{MR1133077}}). 
   %
%   Then $j(\Col{\kappa}{{<}\theta})=\Col{\kappa}{{<}j(\theta)}$ and there is a canonical isomorphism of this partial order and the partial order $\Col{\kappa}{{<}\theta}\times\Col{\kappa}{[\theta,j(\theta))}$ that is an element of $N$. 
  
  Now, let $H_0$ be $\Col{\mu}{[\kappa,j(\kappa))}$-generic over $\VV$. 
  Then there is $H\in\VV[H_0]$ that is $\Col{\mu}{{<}j(\kappa)}$-generic over $\WW$ with $\VV[H_0]=\WW[H]$ and $G\subseteq H$. In this situation, standard arguments (see {\cite[Proposition 9.1]{MR2768691}}) allow us to find an elementary embedding $\map{j_*}{M[G]}{N[H]}$ with $j_*\restriction M=j$. Set $f=j_*(\dot{F}^G)(\kappa)$.
  For any $\gamma < \gamma^\prime < \kappa$, $$\Set{\xi<\delta}{\textit{$\dot{F}^G(\gamma)(\xi)$ is an initial segment of $\dot{F}^G(\gamma^\prime)(\xi)$}}\in U.$$
   So, given $\gamma<\kappa$, elementarity implies that the set $$\Set{\xi<\delta}{\textit{$\dot{F}^G(\gamma)(\xi)$ is an initial segment of $f(\xi)$}}$$ is an element of $U$ since $j_*(\dot{F}^G(\gamma))=\dot{F}^G(\gamma)$. 
   But this implies that $A$ is an initial segment of $[f]_U$ in $\VV[H_0]$ and hence $A$ is not fresh over $\Ult{\VV}{U}$ in $\VV[H_0]$, contradicting Lemma \ref{lemma:ClosedPreservesFreshness}. 
\end{proof}

%A combination of the above results now directly yields a proof of Theorem \ref{theorem:ConsStrength}. 
%
%
%\begin{proof}[Proof of Theorem \ref{theorem:ConsStrength}]
% Theorem \ref{theorem:CollapseWCnefresh} directly shows that the consistency of (i) implies the consistency of (ii), which in turn directly implies the consistency of (iii). 
  %
%  Finally, Theorem \ref{theorem:InnerModelConStrength} allows us to conclude that (iii) implies the existence of an inner model in which (i) holds. 
%\end{proof}

%%%%%%%%%%%%%%%%%%%%%%%%%%%%%%%%%%%%%

\section{Open Questions}

 We end this paper by stating two questions raised by the above results. 

 Our first question is motivated by the fact that, in contrast to the proof of Theorem \ref{theorem:FreshLimitOrdinalsNotCardinals}, our proof of Theorem \ref{theorem:SuccSingularFresh} heavily makes use of the assumption that the $\GCH$ holds at the given singular cardinal. 
 Therefore, it is not possible to use Theorem \ref{theorem:SuccSingularFresh} to derive additional consistency strength from the existence of a normal ultrafilter $U$ on a measurable cardinal $\delta$ and a singular cardinal $\kappa$ of cofinality $\delta$ with the property that no unbounded subset of $\kappa^+$ is fresh over $\Ult{\VV}{U}$, because the existence of a cardinal $\delta<\mu<\kappa$ with $2^\mu>\kappa^+$ might prevent us from applying Theorem \ref{theorem:SuccSingularFresh}, and this constellation can be realized by forcing over a model containing a measurable cardinal. 
  In contrast, if it were possible to remove the $\GCH$ assumption from Theorem \ref{theorem:SuccSingularFresh}, then this would show that the above hypothesis implies that at least one of the following statements holds true: 
  \begin{itemize}
   \item The $\GCH$ fails at a measurable cardinal. 
   
   \item The $\SCH$ fails. 
   
   \item There exists a countably closed singular cardinal $\kappa$ with the property that there are no $\square_\kappa$-sequences. 
  \end{itemize}
  Note that a combination of the main result of \cite{zbMATH04197980}, {\cite[Theorem 1.4]{zbMATH00429348}} and {\cite[Corollary 6]{10.2307/2687750}} shows that the disjunction of the above statements implies the existence of a  measurable cardinal $\kappa$ with $o(\kappa)=\kappa^{++}$ in an inner model. 
  These considerations motivate the following question: 
   
%  In particular, a combination of Theorem \ref{theorem:SuccSingularFresh} with the results of \cite{zbMATH04197980} and \cite{MR2081183} only allows us to derive a measurable cardinal $\kappa$ with $o(\kappa)=\kappa^{++}$ as a lower bound for the consistency strength of the existence of a normal ultrafilter $U$ on a measurable cardinal $\delta$ and a singular strong limit cardinal $\kappa$ of cofinality $\delta$ with the property that no unbounded subset of $\kappa^+$ is fresh over $\Ult{\VV}{U}$. 
  %
%  then this would show that the above statement implies the failure of square at a singular strong limit cardinal and would therefore lead to significantly larger lower bounds for its consistency strength (see \cite{doi:10.1142/S0219061314500032}). 

\begin{question}
  Let $U$ be a normal ultrafilter on a measurable cardinal $\delta$ and let $\kappa$ be a singular  cardinal of cofinality $\delta$ such that $\lambda^\delta<\kappa$ holds for all $\lambda<\kappa$. 
  Assume that there exists a $\square_\kappa$-sequence. Is there an unbounded subset of $\kappa^+$ that is fresh over $\Ult{\VV}{U}$? 
\end{question}

 Our second question addresses the fact that, in the models of set theory studied in Theorems \ref{theorem:sakai} and \ref{theorem:InnerModels}, the existence of fresh subsets only depends on the corresponding measurable cardinal and the cofinality of the given limit ordinal, but not on the specific normal ultrafilter used in the construction of the ultrapower. 
 Therefore, it is natural to ask whether this is always the case.

\begin{question}
 Is it consistent there there exist normal ultrafilters $U_0$ and $U_1$ on a measurable cardinal $\delta$ such that there is a limit ordinal $\lambda$ with the property that no unbounded subset of $\lambda$ is fresh over $\Ult{\VV}{U_0}$ and there exists an unbounded subset of $\lambda$ that is fresh over $\Ult{\VV}{U_1}$? 
\end{question}

%%%%%%%%%%%%%%%%%%%%%%%%%%%%%%%%%%%%%

\bibliographystyle{amsplain}
\bibliography{references}

\end{document}